\documentclass{amsart}

\usepackage{amsmath}
\usepackage{amsthm}
\usepackage{amsfonts}
\usepackage{amssymb}

\newcommand\R{{\mathbf{R}}}

\newcommand\E{{\mathbf{E}}}
\newcommand\Z{{\mathbf{Z}}}

\newcommand\Ent{{\mathbf{H}}}
\renewcommand\P{{\mathbf{P}}}
\newcommand\eps{\varepsilon}
\newcommand\doubling{\sigma}

\newcommand\dist{\operatorname{dist}}
\newcommand\trans{\operatorname{dist}_{\operatorname{tr}}}
\newcommand\range{\operatorname{range}}

\theoremstyle{plain}
  \newtheorem{theorem}[subsection]{Theorem}

  \newtheorem{proposition}[subsection]{Proposition}
  
  \newtheorem{lemma}[subsection]{Lemma}
  \newtheorem{corollary}[subsection]{Corollary}

\theoremstyle{remark}
  \newtheorem{remark}[subsection]{Remark}
  
  \newtheorem{example}[subsection]{Example}

\theoremstyle{definition}
  \newtheorem{definition}[subsection]{Definition}

\include{psfig}

\begin{document}

\title{Sumset and inverse sumset theory for Shannon entropy}

\author{Terence Tao}
\address{Department of Mathematics, UCLA, Los Angeles CA 90095-1555}
\email{tao@math.ucla.edu}

\begin{abstract}  Let $G = (G,+)$ be an additive group.  The sumset theory of Pl\"unnecke and Ruzsa gives several relations between the size of sumsets $A+B$ of finite sets $A, B$, and related objects such as iterated sumsets $kA$ and difference sets $A-B$, while the inverse sumset theory of Freiman, Ruzsa, and others characterises those finite sets $A$ for which $A+A$ is small.  In this paper we establish analogous results in which the finite set $A \subset G$ is replaced by a discrete random variable $X$ taking values in $G$, and the cardinality $|A|$ is replaced by the Shannon entropy $\Ent(X)$.  In particular, we classify the random variable $X$ which have small doubling in the sense that $\Ent(X_1+X_2) = \Ent(X)+O(1)$ when $X_1,X_2$ are independent copies of $X$, by showing that they factorise as $X = U+Z$ where $U$ is uniformly distributed on a coset progression of bounded rank, and $\Ent(Z) = O(1)$.

When $G$ is torsion-free, we also establish the sharp lower bound $\Ent(X+X) \geq \Ent(X) + \frac{1}{2} \log 2 - o(1)$, where $o(1)$ goes to zero as $\Ent(X) \to \infty$.
\end{abstract}

\maketitle

\section{Introduction}

The purpose of this paper is to establish analogues of the Pl\"unnecke-Ruzsa-Freiman sumset and inverse sumset theory for finite subsets of discrete additive groups, in the setting of discrete random variables in such groups.

\subsection{Sumset and inverse sumset theory: a quick review}

To motivate our results we begin by recalling some of the key results in sumset and inverse sumset theory.  Let $G = (G,+)$ be an additive group.  For any finite non-empty sets $A, B$ in $G$, we define the sumset 
$$ A+B := \{a+b: a \in A, b \in B \}$$
and difference set
$$ A-B := \{a-b: a \in A, b \in B \}$$
and the iterated sumsets $2A = A+A$, $3A=A+A+A$, etc.  We use $|A|$ to denote the cardinality of a finite set $A$.

We have the trivial bounds
\begin{equation}\label{atriv}
|A|, |B| \leq |A+B| \leq |A| |B|
\end{equation}
and similarly for $A-B$.  In particular, we see that the \emph{doubling constant}
$$ \doubling[A] := \frac{|A+A|}{|A|}$$
is at least one.  It is easy to see that this doubling constant is precisely one if and only if $A$ is the translate of a finite subgroup of $G$.  Intuitively, one thus expects that if the doubling constant of $A$ is bounded, then $A$ should in some sense \emph{behave like} a translate of a finite subgroup; this is one of the main objectives of the Pl\"unnecke-Ruzsa sumset theory.  One is furthermore interested in \emph{classifying} those sets $A$ of small doubling constant; this is the main objectives of the Freiman-Ruzsa inverse sumset theory.

We now give some sample results in this theory.  One of the simplest is the \emph{Ruzsa triangle inequality}
\begin{equation}\label{ac}
 |A-C| \leq \frac{|A-B| |B-C|}{|B|},
\end{equation}
valid for all non-empty finite subsets $A,B,C$ of $G$ (see e.g. \cite{ruzsa}, \cite[Lemma 2.6]{tao-vu}).  In a similar spirit, one has
\begin{equation}\label{ac2}
|A+B| \leq \frac{|A-B|^3}{|A| |B|}
\end{equation}
(see e.g. \cite{ruzsa}, \cite[Corollary 2.12]{tao-vu}).  If $\doubling[A] \leq K$, then one has the \emph{Pl\"unnecke-Ruzsa inequalities}
\begin{equation}\label{ac3}
|nA-mA| \leq K^{n+m} |A|
\end{equation}
for all $n,m \geq 1$ (see e.g. \cite{ruzsa}, \cite[Corollary 6.28]{tao-vu}).  We refer the reader to \cite{ruzsa} or \cite{tao-vu} for further details of these and related estimates.

Another basic result is the \emph{Balog-Szemer\'edi-Gowers lemma} \cite{balog}, \cite{gowers-4}, which involves \emph{partial sumsets}
$$ A \stackrel{E}{+} B := \{ a+b: (a,b) \in E \}$$
for any subset $E$ of $A \times B$:

\begin{lemma}[Balog-Szemer\'edi-Gowers lemma]\label{bsg}  Suppose that $A, B$ are non-empty finite subsets of an additive group $G$, and let $E \subset A \times B$ be such that $|E| \geq |A| |B|/K$ and $|A \stackrel{E}{+} B| \leq K |A|^{1/2} |B|^{1/2}$ for some $K \geq 1$.  Then there exists subsets $A' \subset A, B' \subset B$ with $|A'| \gg |A|/K$, $|B'| \gg |B|/K$ such that $|A'+B'| \ll K^7 |A'|^{1/2} |B|^{1/2}$.
\end{lemma}

Here and in the sequel, we use $X \ll Y$ or $X = O(Y)$ to denote the estimate $|X| \leq CY$ for some absolute constant $Y$, and $X \asymp Y$ as shorthand for $X \ll Y \ll X$.  If we need the implied constant $C$ to depend in a parameter, we will indicate this by subscripts, thus for instance $O_K(1)$ denotes a quantity bounded in magnitude by $C_K$ for some $C_K$ depending only on $K$.

\begin{proof}  See \cite[Theorem 2.29]{tao-vu}.
\end{proof}

Now we turn to inverse theorems.  A basic concept here is that of a \emph{coset progression}, which unifies the concept of a coset and of an arithmetic progression.

\begin{definition}[Coset progression]\cite{gr-4}  A \emph{coset progression} in an additive group is any set of the form $H+P$, where $H$ is a finite subgroup of $G$, and $P$ is a generalised arithmetic progression, i.e. a set of the form
$$ P := \{ x + n_1 r_1 + \ldots + n_d r_d: n_1 \in [0,N_1),\ldots,n_d \in [0,N_d) \}$$
where $d \geq 0$ is an integer, $x,r_1,\ldots,r_d$ lie in $G$, $N_1,\ldots,N_d \geq 1$ are integers, and $[0,N) := \{0,\ldots,N-1\}$.  We call $d$ the \emph{rank} of the progression.  We say that the coset progression is \emph{$t$-proper} for some $t>0$ if the sums $h+x+n_1r_1+\ldots+n_d r_d$ for $h \in H$ and $n_i \in [0,tN_i)$ are distinct, and \emph{proper} if it is $1$-proper.
\end{definition}

It is easy to see that a coset progression of rank $d$ has doubling constant at most $2^d$.  More generally, if $A$ is a subset of a coset progression $H+P$ with $|A| \geq |H+P|/K$, then $A$ has doubling constant at most $2^d K$.  The following Freiman-type theorem, first proven in \cite{gr-4}, establishes a partial converse to this claim:

\begin{theorem}[Green-Ruzsa Freiman theorem]\label{grf} Let $G$ be an additive group, and let $A \subset G$ be a finite non-empty set with $\doubling[A] \leq K$ for some $K \geq 1$.  Then there exists a coset progression $H+P$ of rank $O(K)$ and size $|H+P| \leq \exp( O( K^{O(1)})) |A|$ such that $A \subset H+P$.
\end{theorem}

\begin{proof} See \cite[Theorem 5.44]{tao-vu}.
\end{proof}

\subsection{Shannon entropy}

We now turn to the concept of Shannon entropy.

\begin{definition}[Shannon entropy]
Let $A$ be a (discrete) set.  Let $\Pr_c(A)$ denote the set of all probability measures on $A$ with compact (i.e. finite) support, or equivalently a function $p: A \to [0,1]$ which is non-zero for only finitely many values, and adds up to one.  Define an \emph{$A$-random variable} to be a random variable $X$ taking values in a finite subset $\range(X) := \{ x \in A: \P(x \in X) \neq 0 \}$, thus the distribution function $p_X(x) := \P( x \in X )$ of $X$ lies in $\Pr_c(A)$.  We write $X \equiv Y$ if $p_X = p_Y$, i.e. if $X, Y$ have the same distribution.  We refer to random variables taking values in a finite set as \emph{discrete random variables}.

The \emph{Shannon entropy} $\Ent(p)$ of a probability distribution $p \in \Pr_c(A)$ is given by the formula
\begin{equation}\label{hbox}
 \Ent(p) := \sum_{x\in A} F( p(x) )
 \end{equation}
where $F: \R^+ \to \R^+$ is the function 
\begin{equation}\label{F-def}
F(x) := x \log \frac{1}{x}
\end{equation}
with the convention that $F(0)=0$.  Given an $A$-random variable $X$, we then define $\Ent(X) := \Ent(p_X)$.  
\end{definition}

The basic theory of Shannon entropy is reviewed in Appendix \ref{shannon-sec}.  For now, we just remark that
\begin{equation}\label{jens}
0 \leq \Ent(X) \leq \log |\range(X)|
\end{equation}
for any discrete random variable $X$, with equality in the former inequality if and only if $X$ is deterministic (i.e. it only takes on one value), and equality in the latter inequality if and only if it is uniformly distributed in $\range(X)$; see Lemma \ref{jensen} for a more precise statement. 
In particular, a \emph{boolean} random variable (i.e. one which takes values in $\{0,1\}$) has entropy at most $\log 2$ with our choice of normalisation of entropy. One can view $G$-random variables as a generalisation of the concept of a finite non-empty subset of $G$, in which the weight (or probability) assigned to each element in the range is not necessarily uniform.

Given two $G$-random variables $X, Y$ (not necessarily independent), their sum $X+Y$ and $X-Y$ are also $G$-random variables, and $\range(X \pm Y) \subset \range(X) \pm \range(Y)$. From standard entropy inequalities one has the trivial upper bound
\begin{equation}\label{ent-sum}
 \Ent(X \pm Y) \leq \Ent(X) + \Ent(Y)
\end{equation}
while if $X$ and $Y$ are independent, one also has the trivial lower bound
\begin{equation}\label{ent-lower}
\Ent(X), \Ent(Y) \leq \Ent(X \pm Y);
\end{equation}
see Lemma \ref{triv}.  The lower bound \eqref{ent-lower} can of course fail if the independence hypothesis is dropped; for instance one clearly has $\Ent(X-X)=0$.

We can define the \emph{doubling constant} $\doubling[X]$ of a $G$-random variable by the formula
$$ \doubling[X] := \exp( \Ent(X_1+X_2) - \Ent(X) )$$
where $X_1,X_2$ are independent copies of $X$, thus $\doubling[X] \geq 1$ by \eqref{ent-lower}.  This quantity is related, but not identical, to the doubling constant $\doubling[A]$ of a set; indeed, from \eqref{jens} we see that 
\begin{equation}\label{doubla}
\doubling[X] \leq \doubling[A]
\end{equation}
whenever $X$ is uniformly distributed on a finite non-empty subset $A$ of $G$.  However, the doubling constant of a random variable can be significantly smaller than that of its range.  For instance, let $A$ be an interval $[0,N)$ together with $\sqrt{N}$ (say) other integers in general position, where $N$ is large.  Then the doubling constant of $A$ is about $\sqrt{N}$, but the uniform distribution on $A$ has doubling constant $O(1)$.  Thus we see that a small amount of ``noise'' (such as the $\sqrt{N}$ integers in general position) can significantly increase the doubling constant of a set, but have only a negligible impact on the doubling constant of a random variable.  Heuristically, one can thus think of entropy sumset theory as a ``noise-tolerant'' analogue of combinatorial sumset theory.

The analogue of a partial sumset $A \stackrel{E}{+} B$ here is the concept of a sum $X+Y$ of \emph{non-independent} random variables $X, Y$.  For instance, if $E \subset A \times B$ is a non-empty set, and $(X,Y)$ is the random variable chosen uniformly at random from $E$, then $X+Y$ is a random variable ranging in $A \stackrel{E}{+} B$.

There are several ways to define the \emph{distance} between two $G$-random variables $X, Y$ (or their associated distributions $p_X, p_Y$).  For instance, we can define their \emph{total variation distance}
\begin{equation}\label{tv-def}
 \dist_{TV}( X, Y ) = \dist_{TV}(p_X, p_Y) := \sum_{x \in G} |p_X(x) - p_Y(x)|;
\end{equation}
this is clearly a metric on $\Pr_c(G)$.  Another useful distance is the \emph{Rusza distance}
\begin{equation}\label{ruzsa-def}
 \dist_R( X, Y ) = \dist_R(p_X, p_Y) := \Ent( X' - Y' ) - \frac{1}{2} \Ent(X') - \frac{1}{2} \Ent(Y')
\end{equation}
where $X', Y'$ are independent copies of $X, Y$ respectively; this is not quite a metric (in particular, $\dist_R(X,X) > 0$ in general), but does obey the triangle inequality as we will see in Theorem \ref{ese} below.  A third distance of importance to us is the following \emph{transport distance}:

\begin{definition}[Transport metric]\label{trans-def}  Let $G$ be an additive group, and let $X, Y$ be $G$-random variables.  We define the \emph{entropy transport distance} $\trans(X,Y)$ from $X$ to $Y$ to be the infimum of $\Ent(Z)$, where $Z$ ranges over all $G$-random variables (not necessarily independent of $X$) such that $X+Z \equiv Y$.
\end{definition}

Observe that $\trans(X,Y)=0$ if and only if $Y$ has the distribution of a translate $X+c$ of $X$.  Up to this equivalence, it is easy to see that the transport distance is indeed a metric.  The notion of transport metric depends only on the distribution, so by abuse of notation we may define $\trans(p_X,p_Y) := \trans(X,Y)$.  The notion of two random variables being close in transport metric is roughly analogous to the notion of (mutual) \emph{$K$-control} of one set by another, introduced in \cite{tao-solvable}.

\begin{example} Let $N$ be a large even integer, let $X$ be the uniform distribution on $[0,N)$, and let $Y$ be the uniform distribution on the even numbers in $[0,N)$.  Then the total variation distance $\dist_{TV}(X,Y)$ is quite large (comparable to its maximal value of $2$).  On the other hand, the Ruzsa distance is quite small (of size $O(1)$).  The transport distance is also of size $O(1)$; indeed, one can transport $Y$ to $X$ by adding a uniform boolean variable $Z \in \{0,1\}$ which is independent of $Y$; conversely, one can transport $X$ to $Y$ by subtracting off the parity bit $Z$ of $X$ (which is \emph{not} independent of $X$).  In fact, the uniform distribution on any dense subset of $[0,N)$ lies within $O(1)$ of $X$ in the transport distance, although this is not as obvious to see; see Corollary \ref{coset-prog} below.
\end{example}

The Ruzsa distance, doubling constant, and transport distance interact well with each other.  For instance, we have the identity
\begin{equation}\label{dubdub}
 \doubling[X] = \exp( \dist_R( X, -X ) )
\end{equation}
and the Lipschitz type properties
\begin{equation}\label{rrt}
 |\dist_R(X',Y') - \dist_R(X,Y)| \leq \frac{3}{2} ( \trans(X,X') + \trans(Y,Y') )
\end{equation}
for any $G$-random variables $X, Y, X', Y'$, as can be seen by several applications of \eqref{ent-sum}.  In particular
\begin{equation}\label{doubtrans}
 |\log \doubling[X] - \log \doubling[X']| \leq 3 \trans(X,X').
\end{equation}
Thus we see that random variables which are close in transport distance are essentially equivalent from the perspective of their sumset theory.

\subsection{Main results}

We can now state our main results.  We begin with some sumset estimates, analogous to \eqref{ac}, \eqref{ac2}, \eqref{ac3}:

\begin{theorem}[Entropy sumset estimates]\label{ese}  Let $G$ be an additive group, and let $X, Y, Z$ be $G$-random variables.
\begin{itemize}
\item (Ruzsa triangle inequality)  We have
\begin{equation}\label{ruzsa-triangle}
 \dist_R(X,Z) \leq \dist_R(X,Y) + \dist_R(Y,Z).
\end{equation}
\item (Sum-difference inequality) One has
\begin{equation}\label{condit}
 \dist_R(X,-Y) \leq 3 \dist_R(X,Y).
\end{equation}
\item (Weak Pl\"unnecke-Ruzsa inequality) If $X_1,\ldots,X_n,X'_1,\ldots,X'_m$ are independent copies of $X$ for some integers $n,m \geq 0$, then
\begin{equation}\label{jest}
\Ent(X_1+\ldots+X_n-X'_1-\ldots-X'_m) \leq \Ent(X) + O( (n+m) \log \doubling[X] ).
\end{equation}
\end{itemize}
\end{theorem}

We prove these estimates in Section \ref{ese-sec}.  The estimate \eqref{jest} loses an absolute constant in comparision to (the logarithm of) \eqref{ac3}.  This is because we do not know how to adapt the graph-theoretic Pl\"unnecke inequality \cite{plun} to the entropy setting, and so must rely instead on some weaker but less deep arguments in \cite{tao-vu} to establish results analogous to \eqref{ac3} instead.

The analogue of the Balog-Szemer\'edi theorem is a little more technical to state, requiring the concept of \emph{conditional entropy} and \emph{conditionally independent trials}, and will be deferred to Section \ref{balog-sec}.

We turn now to an inverse theorem for entropy in the spirit of Theorem \ref{grf}.

\begin{theorem}[Inverse sumset theorem]\label{iest}  Let $G$ be an additive group, and let $X$ be a $G$-random variable.
\begin{itemize}
\item[(i)] $\doubling[X] = 1$ if and only if $X$ is the uniform distribution on a coset of a finite subgroup of $G$.
\item[(ii)] If $\doubling[X] \leq K$, then there exists a coset progression $H+P$ of rank $O_K(1)$ such that $\trans(X,U) \ll_K 1$, where $U$ is the uniform distribution on $H+P$.
\item[(iii)] If $\dist_R(X,Y) \leq K$, then $\trans(X,Y) \ll_K 1$ and $\doubling[X] \leq K^4$.
\end{itemize}
\end{theorem}

Note that the uniform distribution on a coset progression $H+P$ of rank $d$ has doubling constant at most $2^d$, by \eqref{doubla}, so from \eqref{doubtrans} we obtain a partial converse to (ii): if $\trans(X,U) \leq K$ where $U$ is the uniform distribution on a coset progression of rank at most $K$, then $\doubling[X] \ll_K 1$.  Similarly, \eqref{rrt} gives a partial converse to (iii).  Thus, up to constants, Theorem \ref{iest} gives a satisfactory description of random variables with small doubling constant or Ruzsa distance. 

The implied constants in Theorem \ref{iest} can be explicitly computed from the proof, but are rather poor (being triple exponential in $K$).  We will not attempt to optimise these constants here.

We prove Theorem \ref{iest} in Section \ref{iest-sec}.

\subsection{The torsion-free case}

When $G$ is a torsion-free group (thus $nx \neq 0$ for all $x \neq 0$ in $G$ and all integers $n>0$), then the trivial doubling estimate $\doubling[A] \geq 1$ can be improved.  Indeed, one has
$$ |A+A| \geq 2|A|-1$$
for any finite non-empty subset $A$ of a torsion-free group $G$, since $A$ can be mapped onto the integers by a Freiman isomorphism (see \cite[Lemma 5.25]{tao-vu}).  In other words, one has
\begin{equation}\label{doub}
\doubling[A] \geq 2 - \frac{1}{|A|}.
\end{equation}
The example of an arithmetic progression (e.g $A = [0,n)$) shows that this estimate is sharp.

One can ask whether the same statement holds for entropy.  The following example shows that this is not quite the case.  Let $n$ be a large integer, and let $X_n$ be the sum of $n$ independent Bernoulli variables $\epsilon_1,\ldots,\epsilon_n \in \{-1,+1\}$ with an equal probability of each.  From the central limit theorem (or Stirling's formula), we know that $X_n$ is approximately distributed like a gaussian of mean zero and variance $n$, thus
$$ p_{X_n}(m) \approx \frac{1}{\sqrt{2\pi n}} e^{-m^2/2n}.$$
Approximating the Riemann sum by an integral, we then expect
$$ \Ent(X_n) \approx \int_\R F( \frac{1}{\sqrt{2\pi n}} e^{-x^2/2n} )\ dx = \log \sqrt{2\pi n} + \frac{1}{2}.$$
It is not hard to make this heuristic precise, and obtain the asymptotic
$$\Ent(X_n) = \log \sqrt{2\pi n} + \frac{1}{2} + o(1).$$
In particular, since $X_n + X'_n \equiv X_{2n}$ if $X'_n$ is an independent copy of $X_n$, we see that
$$ \doubling[X_n] = \sqrt{2} - o(1),$$
which is less than what one might have predicted from \eqref{doub}, \eqref{doubla}.  The point is that in the entropy setting one can construct ``approximate gaussian'' counterexamples whose closest analogue in the combinatorial setting, namely the arithmetic progressions, are less efficient by a constant factor.

It should not be surprising to experts in information theory that this gaussian-type bound is best possible:

\begin{theorem}\label{abb}  If $\eps > 0$, $G$ is torsion-free, and $X$ is a $G$-random variable, then
$$ \doubling(X) \geq \sqrt{2}-\eps,$$
provided $\Ent(X)$ is sufficiently large depending on $\eps$.
\end{theorem}

In asymptotic notation, Theorem \ref{abb} asserts that
$$ \doubling(X) \geq \sqrt{2} - o_{\Ent(X) \to \infty}(1).$$

We prove Theorem \ref{abb} in Section \ref{abb-sec} below.  This result combines the inverse theorem in Theorem \ref{iest} with an analogous inequality concerning the Shannon entropy 
$$\Ent_\R(X) :=  \int_\R F(p_X(x))\ dx$$ 
of a \emph{continuous} random variable $X$ taking values of $\R$, namely
\begin{equation}\label{st}
\Ent_\R( S + T ) \geq \frac{1}{2}(\Ent_\R(S) + \Ent_\R(T)) + \frac{1}{2} \log 2
\end{equation}
for independent continuous random variables $S, T$ (see \cite[Theorem 2]{abbn}).  The inverse sumset theory is necessary in order to approximate the discrete random variable by a continuous one in a certain sense.

In \cite{abbn}, the continuous entropy inequality
$$ \Ent_\R(X_1+\ldots+X_{n+1}) \geq \Ent_\R(X_1+\ldots+X_n) + \log \frac{\sqrt{n+1}}{\sqrt{n}}$$
was established, where $X_1,\ldots,X_{n+1}$ were independent copies of the same continuous random variable.  In view of Theorem \ref{abb}, it is thus natural to conjecture that
\begin{equation}\label{entxx}
 \Ent(X_1+\ldots+X_{n+1}) \geq \Ent(X_1+\ldots+X_n) + \log \frac{\sqrt{n+1}}{\sqrt{n}} - \eps
\end{equation}
for any $\eps > 0$ and any $G$-random variable $X$, if $G$ is torsion-free and $\Ent(X)$ is sufficiently large depending on $n$, $\eps$.  Unfortunately we were not able to establish this because the inverse theorem is not applicable in this setting, nevertheless we believe \eqref{entxx} to be true.

Finally, we remark that a number of additional entropy sumset inequalities were recently established in \cite{mmt}.  For instance, it was shown that
$$ \Ent(X+Y+Z) \leq \frac{1}{2} (\Ent(X+Y)+\Ent(Y+Z)+\Ent(Z+X))$$
for independent $G$-random variables $X,Y,Z$, which is an entropy analogue of the inequality
$$ |A+B+C| \leq |A+B|^{1/2} |B+C|^{1/2} |C+A|^{1/2}$$
(see e.g. \cite{gmr} or \cite{bela} for a proof). However, these bounds are primarily of interest in the regime where the doubling constants of the sets involved are large, and so are not directly related to the ones presented here.

\subsection{Acknowledgments}

The author is supported by NSF Research Award DMS-0649473, the NSF Waterman award and a grant from the MacArthur Foundation.  
The material in Sections \ref{ese-sec}, \ref{balog-sec} are based on some unpublished notes of the author with Van Vu.  We are indebted to the anonymous referee for a careful reading of the paper and many useful corrections and suggestions.

\section{Sumset estimates}\label{ese-sec}

In this section we establish the various sumset estimates claimed in the introduction, and in particular establish Theorem \ref{ese}.  The main tools will be entropy inequalities (in particular the submodularity inequality, Lemma \ref{submodularity}), elementary arithmetic identities, and independent and conditionally independent trials.

Readers who are familiar with the combinatorial analogues of these inequalities are invited to ``pretend'' that all of the random variables below are uniformly distributed on various finite sets, and in particular on finite groups, in order to see the analogy between both the statements and the proofs of the combinatorial and the entropy estimates.  Indeed, the arguments here were discovered by the reverse of this procedure, in which the author searched for the nearest entropy-theoretic analogue to each step in the combinatorial arguments.  For instance, if the combinatorial argument required one to pick an object $a$ from a finite set $A$, the entropy-based argument would instead consider an analogous random variable that was naturally associated to $A$; if the combinatorial argument required two objects to be related in some way, this usually manifested itself as a coupling of random variables (e.g. by the use of conditionally independent trials); and so forth.

We begin with the trivial sum set estimates.

\begin{lemma}[Trivial sumset estimate]\label{triv}  If $X, Y$ are two $G$-random variables, and $Z$ is a discrete random variable, then
$$ \Ent(X+Y|Z) \leq \Ent(X|Z) + \Ent(Y|Z).$$
If furthermore $X, Y$ are conditionally independent relative to $Z$, then
$$
\max(\Ent(X|Z), \Ent(Y|Z)) \leq \Ent(X+Y|Z).$$
In particular we have the inequalities \eqref{ent-sum}, \eqref{ent-lower}, and $\dist_R(X,Y) \geq 0$ for all $G$-random variables $X,Y$.
\end{lemma}

\begin{proof} By conditioning on $Z$ we may assume that $Z$ is deterministic, thus the task reduces to showing \eqref{ent-sum} and \eqref{ent-lower}.  The former inequality follows from \eqref{ent-sum-} and \eqref{entyx} since $(X,Y)$ determines $X+Y$.    To prove the latter inequality, observe 
from \eqref{ento}, \eqref{fsqueeze}, and the independence of $X,Y$ that
$$ \Ent(X+Y) \geq \Ent(X+Y|Y) = \Ent(X|Y) = \Ent(X)$$
and similarly $\Ent(X+Y) \geq \Ent(Y)$, and the claim follows.
\end{proof}

Now we establish the Ruzsa triangle inequality \eqref{ruzsa-triangle}, which we rewrite as
$$ \Ent(X-Z) \leq \Ent(X-Y) + \Ent(Y-Z) - \Ent(Y)$$
where $X,Y,Z$ are independent $G$-random variables.  Observe that $(X-Y,Y-Z)$ and $(X,Z)$ both determine $X-Z$, while $(X-Y,Y-Z)$ and $(X,Z)$ jointly determine $(X,Y,Z)$.  By the submodularity inequality (Lemma \ref{submodularity}) we conclude that
$$ \Ent(X,Y,Z) + \Ent(X-Z) \leq \Ent(X-Y,Y-Z) + \Ent(X,Z).$$
Applying \eqref{ent-sum-} and the independence hypotheses we obtain the claim.

To prove \eqref{condit}, we introduce the idea of \emph{conditionally independent trials}.  Given two random variables $X, Y$ (not necessarily independent), we can produce two conditionally independent trials $X_1, X_2$ of $X$ relative to $Y$, defined by declaring $(X_1|Y=y)$ and $(X_2|Y=y)$ to be independent trials of $(X|Y=y)$ for all $y \in \range(Y)$, thus in particular $X_1 \equiv X_2 \equiv X$, and $X_1,X_2$ are conditionally independent relative to $Y$.  Observe from conditional independence that
$$ \Ent(X_1,X_2|Y) = \Ent(X_1|Y) + \Ent(X_2|Y) = 2\Ent(X|Y)$$
and thus
\begin{equation}\label{entity}
 \Ent(X_1,X_2,Y) = 2 \Ent(X,Y) - \Ent(Y).
 \end{equation}

Let $X, Y$ be independent $G$-random variables.  Let $(X_1,Y_1), (X_2,Y_2)$ be conditionally independent trials of $(X,Y)$ relative to $X-Y$; since $(X,Y)$ determines $X-Y$, we conclude that $X_1-Y_1=X_2-Y_2$.  Let $(X_3,Y_3)$ be another trial of $(X,Y)$, independent of $X_1,X_2,Y_1,Y_2$, then we have the identity
$$ X_3+Y_3 = (X_3-Y_2) - (X_1-Y_3) + X_2 + Y_1.$$
Thus $(X_3-Y_2,X_1-Y_3,X_2,Y_1)$ and $(X_3,Y_3)$ each determine $X_3+Y_3$, while $(X_3-Y_2,X_1-Y_3,X_2,Y_1)$ and $(X_3,Y_3)$ together determine $(X_1,X_2,X_3,Y_1,Y_2,Y_3)$; applying the submodularity inequality (Lemma \ref{submodularity}) we conclude
$$ \Ent(X_1,X_2,X_3,Y_1,Y_2,Y_3) + \Ent(X_3+Y_3) \leq \Ent(X_3-Y_2,X_1-Y_3,X_2,Y_1) + \Ent( X_3, Y_3 ).$$
But from \eqref{entity}, \eqref{ent-sum-}, and the independence hypotheses we have
\begin{align*}
\Ent(X_1,X_2,X_3,Y_1,Y_2,Y_3) &= 2 \Ent(X,Y) - \Ent(X-Y) + \Ent(X) + \Ent(Y) \\
\Ent(X_3+Y_3) &= \Ent(X+Y) \\
\Ent(X_3-Y_2,X_1-Y_3,X_2,Y_1) &\leq 2 \Ent(X-Y) + \Ent(X) + \Ent(Y) \\
\Ent(X_3,Y_3) &= \Ent(X)+\Ent(Y)
\end{align*}
and thus
\begin{equation}\label{entpm}
 \Ent(X+Y) \leq 3\Ent(X-Y) - \Ent(X) - \Ent(Y)
\end{equation}
which rearranges to form \eqref{condit}.

Finally, we establish \eqref{jest}.  Let $X, Y$ be independent $G$-random variables, and let $(X_0,Y_0),\ldots,(X_n,Y_n)$ be independent trials of $(X,Y)$.  Set $S_i = X_i + Y_i$ for $0 \leq i \leq n$.  We observe the identity
$$ S_0 + \ldots + S_n = (Y_0+X_1) + (Y_1+X_2) + \ldots + (Y_{n-1}+X_n) + (Y_n + X_0).$$
In particular, we see that $(X_0,Y_0,S_1,\ldots,S_n)$ and $(Y_0+X_1,\ldots,Y_{n-1}+X_n,Y_n+X_0)$ both determine $S_0+\ldots+S_n$, while $(X_0,Y_0,S_1,\ldots,S_n)$ and $(Y_0+X_1,\ldots,Y_{n-1}+X_n,Y_n+X_0)$ jointly determine $(X_0,\ldots,X_n,Y_0,\ldots,Y_n)$.  Applying the submodularity inequality (Lemma \ref{submodularity}) we conclude
$$ \Ent(X_0,\ldots,X_n,Y_0,\ldots,Y_n) + \Ent(S_0+\ldots+S_n) \leq \Ent(X_0,Y_0,S_1,\ldots,S_n) + \Ent(Y_0+X_1,\ldots,Y_{n-1}+X_n,Y_n+X_0).$$
But from \eqref{ent-sum-} and the independence hypotheses we have
\begin{align*}
\Ent(X_0,\ldots,X_n,Y_0,\ldots,Y_n) &= (n+1) (\Ent(X) + \Ent(Y)) \\
\Ent(X_0,Y_0,S_1,\ldots,S_n) &= \Ent(X) + \Ent(Y) + n \Ent(X+Y) \\
\Ent(Y_0+X_1,\ldots,Y_{n-1}+X_n,Y_n+X_0) &\leq (n+1) \Ent(X+Y);
\end{align*}
Putting all this together we obtain the inequality
$$ \Ent(S_0+\ldots+S_n) \leq (2n+1) \Ent(X+Y) - n \Ent(X) - n \Ent(Y).$$
In particular, if $X_1,\ldots,X_{2n+2}$ are independent copies of $X$ then the above inequality (setting $Y$ to be another independent copy of $X$) gives
$$ \Ent(X_1+\ldots+X_{2n+2}) \leq \Ent(X) + (2n+1) \log \doubling[X];$$
applying \eqref{ent-lower} one concludes that
$$ \Ent(X_1+\ldots+X_n) = \Ent(X) + O( n \log \doubling[X] )$$
for any $n \geq 1$.  Applying \eqref{entpm} one then concludes that
$$ \Ent(X_1+\ldots+X_n - X'_1 - \ldots - X'_m) = \Ent(X) + O( (n+m) \log \doubling[X] )$$
for any $n,m \geq 1$, and the claim \eqref{jest} follows.  The proof of Theorem \ref{ese} is now complete.

\section{An entropy version of the Balog-Szemer\'edi-Gowers lemma}\label{balog-sec}

We now state an entropy analogue of the Balog-Szemer\'edi-Gowers lemma.  In the combinatorial setting, one had the notion of a \emph{refinement} $A'$ of a set $A$, which was a subset $A'$ of $A$ which still had size comparable to $A$.  In the entropy setting, the corresponding notion is that of a \emph{conditioning} of a random variable $X$ relative to some other related random variable $Y$, such that $\Ent(X|Y)$ was still close to $\Ent(X)$.  The entropy Balog-Szemer\'edi-Gowers lemma then asserts that if two weakly dependent random variables $X, Y$ have a sum of small entropy, then there exist conditionings of $X, Y$ (which capture most of the entropy) whose \emph{independent} sum still has small entropy.

In fact, the conditioning can be given explicitly:

\begin{theorem}[Entropy Balog-Szemer\'edi-Gowers lemma]\label{ent-bsg}  Let $G$ be an additive group, and let $X, Y$ be $G$-random variables which are weakly dependent in the sense that
\begin{equation}\label{entk}
 \Ent(X,Y) \geq \Ent(X) + \Ent(Y) - \log K
 \end{equation}
for some $K \geq 1$.  Suppose also that
\begin{equation}\label{entk2}
 \Ent(X+Y) \leq \frac{1}{2} \Ent(X) + \frac{1}{2} \Ent(Y) + \log K.
 \end{equation}
Then if we let $(X_1,Y), (X_2,Y)$ be conditionally independent trials of $(X,Y)$ conditioning on $Y$, and then let $(X_1, X_2, Y)$ and $(X_1, Y')$ be conditionally independent trials of $(X_1,X_2,Y)$ and $(X_1,Y)$ conditioning on $X_1$, then $X_2$ and $Y'$ are conditionally independent relative to $X_1, Y$, with
\begin{align}
\Ent( X_2 | X_1, Y ) &\geq \Ent(X) - \log K \label{loga}\\
\Ent( Y' | X_1, Y ) &\geq \Ent(Y) - \log K \label{logb}\\
\Ent( X_2 + Y' | X_1, Y ) &\leq \frac{1}{2} \Ent(X) + \frac{1}{2} \Ent(Y) + 7 \log K.\label{seven}
\end{align}
\end{theorem}

This should be compared with Lemma \ref{bsg}.  The appearance of the exponent $7$ in both statements is not coincidental, as the proofs are fundamentally the same.

\begin{remark} Let $E \subset A \times B$ be a regular bipartite graph between two finite non-empty sets $A, B$ in $G$, thus the $A$-degree $|\{ b \in B: (a,b) \in E \}|$ is independent of $a \in A$, and similarly the $B$-degree $|\{ a \in A: (a,b) \in E \}|$ is independent of $b \in B$.  Let $(X,Y)$ be an element of $E$ chosen uniformly at random.  Then the random variables $(Y', X_1, Y, X_2)$ defined above are drawn uniformly from the space of all paths $(b, a, b', a')$ of length three in $E$, thus $(a,b), (a,b'), (a',b') \in E$.  It may be helpful to keep this example in mind when going through the proof of Theorem \ref{ent-bsg}.  Not surprisingly, paths of length three also play a major role in the proof of Theorem \ref{bsg}.
\end{remark}

We now establish the theorem.  By construction, $Y'$ and $X_2,Y$ are conditionally independent relative to $X_1$, and thus $X_2$ and $Y'$ are conditionally independent relative to $X_1,Y$ as claimed.  Also, since $X_1$ is conditionally independent of $X_2$ relative to $Y$, one has
$$ \Ent(X_2|X_1,Y) = \Ent(X_2|Y) = \Ent(X|Y)$$
and \eqref{loga} follows from \eqref{entk}.  Similarly, since $Y, Y'$ are conditionally independent relative to $X_1$, one has
$$ \Ent(Y'|X_1,Y) = \Ent(Y'|X_1) = \Ent(Y|X)$$
and \eqref{logb} follows from \eqref{entk}.

The only remaining claim to establish is \eqref{seven}.  We need a preliminary lemma:

\begin{lemma}[Weak Balog-Szemer\'edi-Gowers lemma]\label{weak-bsg} We have
$$ \Ent(X_1-X_2|Y) \leq \Ent(X) + 4 \log K.$$
\end{lemma}

\begin{proof}  Let $(X_1,X_2,Y), (X_1,X_2,Y')$ be two conditionally independent copies of $(X_1,X_2,Y)$ relative to $(X_1,X_2)$.  Observe that $(X_1,X_2,Y)$, $(X_1+Y', X_2+Y',Y)$ both determine $(X_1-X_2, Y)$, and that $(X_1,X_2,Y)$ and $(X_1+Y', X_2+Y',Y)$ jointly determine $(X_1,X_2,Y,Y')$.  Applying the submodularity inequality (Lemma \ref{submodularity}) we conclude that
$$ \Ent( X_1,X_2,Y,Y' ) + \Ent(X_1-X_2, Y) \leq \Ent(X_1,X_2,Y) + \Ent(X_1+Y', X_2+Y',Y).$$
But from \eqref{ent-sum}, \eqref{entity}, \eqref{eident}, one has
\begin{align*}
\Ent( X_1,X_2,Y,Y' ) &= 2 \Ent(X_1,X_2,Y) - \Ent(X_1,X_2) \\
&\geq 4 \Ent(X,Y) - 2 \Ent(Y) - 2 \Ent(X) \\
\Ent(X_1-X_2,Y) &= \Ent(X_1-X_2|Y) + \Ent(Y) \\
\Ent(X_1,X_2,Y) &= 2 \Ent(X,Y) - \Ent(Y) \\
\Ent(X_1+Y',X_2+Y',Y) &\leq 2 \Ent(X+Y) + \Ent(Y)
\end{align*}
and thus
$$ \Ent(X_1-X_2|Y) \leq 2\Ent(X+Y) + \Ent(Y) + 2 \Ent(X) - 2 \Ent(X,Y),$$
and the claim then follows from \eqref{entk}, \eqref{entk2}.
\end{proof}

Now observe that $(X_2, Y', Y)$ and $(X_1-X_2,X_1+Y', Y)$ both determine $(X_2+Y',Y)$, and that $(X_2, Y', Y)$ and $(X_1-X_2,X_1+Y', Y)$ jointly determine $(X_1,X_2,Y,Y')$.  Applying the submodularity inequality (Lemma \ref{submodularity}) we conclude that
$$ \Ent(X_1,X_2,Y,Y') + \Ent(X_2+Y',Y) \leq \Ent(X_2, Y', Y) + \Ent(X_1-X_2,X_1+Y', Y).$$
But from \eqref{ent-sum}, \eqref{eident}, and (a generalisation of) \eqref{entity}, one has
\begin{align*}
\Ent(X_1,X_2,Y,Y') &= \Ent(X_1,X_2,Y) + \Ent(X_1,Y') - \Ent(X_1) \\
&= 2\Ent(X,Y) - \Ent(Y) + \Ent(X,Y) - \Ent(X) \\
\Ent(X_2+Y',Y) &= \Ent(X_2+Y'|Y) + \Ent(Y) \\
\Ent(X_2,Y',Y) &\leq \Ent(X_2,Y) + \Ent(Y') \\
&= \Ent(X,Y) + \Ent(Y) \\
\Ent(X_1-X_2,X_1+Y',Y) &\leq \Ent(X_1-X_2|Y) + \Ent(Y) + \Ent(X_1+Y') \\
&= \Ent(X_1-X_2|Y) + \Ent(Y) + \Ent(X+Y)
\end{align*}
Substituting these bounds, we obtain
$$ \Ent(X_2+Y'|Y) \leq \Ent(X_1-X_2|Y) + 2\Ent(Y) + \Ent(X) + \Ent(X+Y) - 2\Ent(X,Y).$$
Applying Lemma \ref{weak-bsg}, \eqref{entk}, \eqref{entk2} we conclude that
$$ \Ent(X_2+Y'|Y) \leq \frac{1}{2} \Ent(X) + \frac{1}{2} \Ent(Y) + 7 \log K$$
and the claim follows from \eqref{entyx}.

\section{Uniformisation}

The main purpose of this section is to establish the following uniformisation bound on groups, as well as an analogous result for coset progressions (see Corollary \ref{coset-prog}).

\begin{theorem}[Uniformisation on a group]\label{uni-thm} Let $G$ be a finite group, let $p_U := \frac{1}{|G|}$ be the uniform distribution on $G$, and let $p \in \Pr_c(G)$ be another distribution, such that
$$ \Ent(p) \geq \log |G| - \log K$$
for some $K \geq 10$.  Then 
$$
\trans(p,p_U) \ll \log K.
$$
\end{theorem}

Since $\Ent(p_U) = \log |G|$, we see that this is sharp up to constants.  One can view this theorem as a special case of Theorem \ref{iest}, but with significantly better dependence on the constants.

We establish this theorem by a sequence of partial results.  We first record a simple lemma that allows us to ``divide and conquer'' the problem of estimating the transport distance between two random variables.

\begin{lemma}[Transport splitting lemma]\label{trans-split}  Let $G$ be a group, let $X, Y$ be $G$-random variables, and let $S$ be another discrete random variable; we do not assume $X,Y,S$ to be independent.  Then
$$ \trans(X,Y) \leq \Ent(S) + \sum_{s \in \range(S)} p_S(s) \trans( (X|S=s), (Y|S=s) ).$$
\end{lemma}

\begin{proof}  Let $\eps > 0$.  For each $s \in \range(S)$, we can use Definition \ref{trans-def} to select a random variable $Z_s$ conditioned to the event $S=s$ of entropy $\Ent(Z_s) \leq \trans( (X|S=s), (Y|S=s) ) + \eps$ such that $(X+Z_s|S=s) \equiv (Y|S=s)$.  If we then let $Z$ be the random variable whose conditioning to $S=s$ equals $Z_s$, then $X+Z \equiv Y$, and (by \eqref{ento})
$$ \Ent(Z) \leq \Ent(S) + \Ent(Z|S) = \Ent(S) + \sum_{s \in \range(S)} p_S(s) \Ent(Z|S=s).$$
Putting all this together, we conclude that
$$ \trans(X,Y) \leq \Ent(S) + \sum_{s \in \range(S)} p_S(s) \trans( (X|S=s), (Y|S=s) ) + |\range(S)| \eps.$$
Since $\eps$ was arbitrary, the claim follows.
\end{proof}

Next, we show that one can converge exponentially fast to the uniform distribution in the $L^2$ sense.

\begin{lemma}[$L^2$ flattening lemma]\label{lflat}  Let $G$ be a finite group, let $p_U := \frac{1}{|G|}$ be the uniform distribution on $G$, and let $p \in \Pr_c(G)$ be another distribution.  Then for any integer $k \geq 1$, one can find a distribution $p_k \in \Pr_c(G)$ such that
$$ \trans( p, p_k ) \leq k \log 2$$
and
$$ \| p_k - p_U \|_{\ell^2(G)} \leq 2^{-k/2} \| p - p_U \|_{\ell^2(G)}.$$
\end{lemma}

\begin{proof}  By induction it suffices to verify the case $k=1$.  We use the first moment method.  Let $h$ be chosen uniformly at random from $G$, and let
$$ p_1(x) := \frac{1}{2} (p(x) + p(x-h)).$$
Clearly $p_1$ is the convolution of $p$ with a Bernoulli variable of entropy $\log 2$, and so $\trans(p,p_1) \leq \log 2$.  On the other hand, a straightforward calculation using $\sum_{x \in G} p(x)=1$ reveals the identity
$$ \E_h \| p_1 - p_U \|_{\ell^2(G)}^2 = \frac{1}{2} \| p - p_U \|_{\ell^2(G)}^2$$
and the claim follows.
\end{proof}

We now combine these lemmas to pass to an $\ell^2$-bounded random variable.

\begin{lemma}[Entropy-uniform to $\ell^2$-bounded]\label{entl2}  Let $G$ be a finite group, and let $p \in \Pr_c(G)$ be such that
$$ \Ent(p) \geq \log |G| - \log K$$
for some $K \geq 10$.  Then there exists $q \in \Pr_c(G)$ with
$$
\trans(p,q) \ll \log K
$$
and
$$
\| q \|_{\ell^2(G)} \ll 1/|G|^{1/2}.$$
\end{lemma}

\begin{proof}  The basic idea here is to flatten all the regions of $G$ in which $X$ has an abnormally high probability density.

By Lemma \ref{jensen}, we have
\begin{equation}\label{gack}
 \sum_{k=1}^\infty 2^k \P( X \in A_k ) \ll \log K
\end{equation}
where $A_k$ are the sets 
$$ A_k := \{ x \in G: \frac{2^{2^{k-1}}}{|G|} \leq p(x) < \frac{2^{2^k}}{|G|} \}$$
for $k \geq 1$, and then set $A_0 := G \backslash \bigcup_{k=1}^\infty A_k$, thus the $A_0, A_1, \ldots$ partition $G$ (and thus only finitely many are non-empty).

Let $X$ be a random variable with distribution $p$.
For each $k \geq 0$, let $E_k$ be the event that $X \in A_k$, thus the $E_k$ partition probability space.  From \eqref{gack} we have
\begin{equation}\label{gack-2}
\sum_{k=1}^\infty 2^k \P(E_k) \ll \log K.
\end{equation}

Suppose that $k \geq 1$ is such that $E_k$ has positive probability.  Then we can define $X_k$ to be the random variable $X_k := (X|E_k)$.  Observe that $p_{X_k}$ is bounded above by $\frac{2^{2^k}}{\P(E_k) |G|}$, so we have the crude bound
$$ \| p_{X_k} \|_{\ell^2(G)} \leq \frac{2^{2^k}}{\P(E_k) |G|^{1/2}}.$$
Applying Lemma \ref{lflat} (with $k$ replaced by (say) $2 \times (2^k + \log \frac{1}{\P(E_k)})$), one can thus find $q_k \in \Pr_c(G)$ such that
\begin{equation}\label{soso}
 \trans(p_{X_k}, q_k) \ll 2^k + \log \frac{1}{\P(E_k)}
 \end{equation}
and
\begin{equation}\label{qkpu}
 \| q_k - p_U \|_{\ell^2(G)} \leq 1 / |G|^{1/2}.
\end{equation}
(Indeed, one could even gain a factor of $2^{-2^k}$ on the right-hand side of \eqref{qkpu}, though this turns out to be unnecessary for the current argument.)

Now set $q \in \Pr_c(G)$ to be the probability distribution
\begin{equation}\label{pxy}
 q = 1_{E_0} p_X + \sum_{k=1}^\infty \P( E_k ) q_k .
\end{equation}
Observe that $p_X$ is bounded by $2/|G|$ on $E_0$.  From \eqref{pxy}, \eqref{qkpu} and the triangle inequality we conclude that
$$ \| q \|_{\ell^2(G)} \ll 1/|G|^{1/2}.$$
From Lemma \ref{trans-split} (setting $S$ to be the random variable induced by the partition $E_k$), we see that
\begin{equation}\label{entys}
\trans(p,q) \leq \sum_{k=0}^\infty \P(E_k) \log \frac{1}{\P(E_k)} + \sum_{k=1}^\infty \P(E_k) \trans(p_{X_k}, q_k).
\end{equation}
From \eqref{entys}, \eqref{soso}, \eqref{gack-2} we conclude that
$$ \trans(p,q) \ll \log K + \sum_{k=0}^\infty \P(E_k) \log \frac{1}{\P(E_k)}.$$
But from \eqref{gack-2}, $\P(E_k) \ll 2^{-k} \log K$, and so 
$$\P(E_k) \log \frac{1}{\P(E_k)} \ll (1+k) 2^{-k} \log K,$$
and thus
$$ \sum_{k=0}^\infty \P(E_k) \log \frac{1}{\P(E_k)} \ll \log K$$
and the claim follows.
\end{proof}

From the triangle inequality, it is now clear that Theorem \ref{uni-thm} follows from Lemma \ref{entl2}, Lemma \ref{lflat}, and

\begin{lemma}[$\ell^2$-bounded to uniform]  Let $G$ be a finite group, and let $p \in \Pr_c(G)$ be such that $\| p - p_U \|_{\ell^2(G)} \leq 1/|G|^{1/2}$.  Then $\trans(p, p_U) \ll 1$.
\end{lemma}

\begin{proof}  The idea is to manually transport away the most severe irregularities in the distribution of $p$ to obtain a new distribution that is significantly closer to uniform in the $\ell^2$ norm, and then iterate.

Let $C_G$ be the supremum of $\trans(p,p_U)$ for all $p \in \Pr_c(G)$ with $\| p - p_U \|_{\ell^2(G)} \leq 1/|G|^{1/2}$.  It is easy to see that $C_G$ is finite for any fixed finite $G$; our task is to show that $C_G \ll 1$ (uniformly in $G$).

Let $k \geq 1$ be a large integer to be chosen later.  Let $p \in \Pr_c(G)$ be such that $\| p - p_U \|_{\ell^2(G)} \leq 1/|G|^{1/2}$, then by Lemma \ref{lflat} one can find $q \in \Pr_c(G)$ with $\trans(p,q) \leq k \log 2$ and
\begin{equation}\label{kewpie}
\| q - p_U \|_{\ell^2(G)} \leq 2^{-k/2} / |G|^{1/2}.
\end{equation}
If $q \equiv p_U$, then we have $\trans(p,p_U) \leq k \log 2$, so suppose instead that $q$ is not identically equal to $p_U$.  Then the quantity
$$ \sigma := \sum_{x \in G: q(x) > p_U(x)} q(x) - p_U(x) = \sum_{x \in G: q(x) < p_U(x)} p_U(x) - q(x)$$
is non-zero; from \eqref{kewpie} and the Cauchy-Schwarz inequality we also have
\begin{equation}\label{sigma-bound}
\sigma < 2^{-k/2}.
\end{equation}

Let $q_+, q_- \in \Pr_c(G)$ be the probability distributions defined by
$$ q_+(x) := \frac{1}{\sigma} 1_{q(x) > p_U(x)} (q(x) - p_U(x))$$
and
$$ q_-(x) := \frac{1}{\sigma} 1_{q(x) < p_U(x)} (p_U(x) - q(x)),$$
thus
$$ q = p_U + \sigma q_+ - \sigma q_-.$$
We can thus build a random variable $X$ with distribution $q$ by creating a boolean random variable $S \in \{0,1\}$ with $p_S(1) = \sigma$, then setting $(X|S=1)$ to have distribution $q_+$ and $(X|S=0)$ to have distribution $\frac{1}{1-\sigma} (p_U - \sigma q_-)$.  If we let $Y$ be a random variable with $(Y|S=1)$ having distribution $q_-$ and $(Y|S=0)$ having distribution $\frac{1}{1-\sigma} (p_U - \sigma q_-)$, we see that $Y \equiv p_U$.  From Lemma \ref{trans-split} we conclude that
$$ \trans(q,p_U) \leq \Ent(S) + \sigma \trans(q_+,q_-) \ll \sigma \log \frac{1}{\sigma} + \sigma \trans(q_+,q_-).$$
Now we estimate $\trans(q_+,q_-)$.  From \eqref{kewpie} we see that
$$ \| q_+ \|_{\ell^2(G)}, \|q_-\|_{\ell^2(G)} \ll \frac{2^{-k/2}}{\sigma |G|^{1/2}}.$$
Applying Lemma \ref{lflat}, one can find $r_+, r_- \in \Pr_c(G)$ with $\|r_\pm - p_U \|_{\ell^2(G)} \leq 1/|G|^{1/2}$ such that
$$ \trans(q_\pm, r_\pm) \ll 1 + \log \frac{2^{-k/2}}{\sigma}.$$
By definition of $C_G$, we then have
$$ \trans(r_\pm, p_U) \ll C_G.$$
Putting all this together using the triangle inequality, we see that
$$ \trans(p,p_U) \ll k + \sigma \log \frac{1}{\sigma} + \sigma( C_G + \log \frac{2^{-k/2}}{\sigma} ).$$
Taking the worst-case value of $\sigma$ using \eqref{sigma-bound} we conclude
$$ \trans(p,p_U) \ll k + 2^{-k/2} C_G $$
and thus on taking suprema in $p$
$$ C_G \ll k + 2^{-k/2} C_G.$$
Setting $k$ sufficiently large we conclude
$$ C_G \leq \frac{1}{2} C_G + O(1)$$
and the claim follows.
\end{proof}

\begin{corollary}[Uniformisation on coset progressions]\label{coset-prog}  Let $H+P$ be a proper coset progression of rank $d$ in some additive group $G$, and let $p \in \Pr_c(H+P)$ be such that
$$ \Ent(p) \geq \log |H+P| - \log K$$
for some $K \geq 10$.  Let $p_U$ be the uniform distribution on $H+P$.  Then
\begin{equation}\label{transpu}
\trans(p, p_U) \ll \log K + d.
\end{equation}
\end{corollary}

\begin{proof}  We can view $H+P$ as the homomorphic image of $B := H \times [0,N_1) \times \ldots \times [0,N_d)$ for some integers $N_1,\ldots,N_d \geq 1$.  Let $\tilde p \in \Pr_c(B)$ be the pullback of $p$ to $B$, and similarly define $\tilde p_U$.  We can then embed $B$ in the finite group $G := H \times \Z/(2N_1\Z) \times \ldots \times \Z/(2N_d\Z)$.  Observe that
$$ \Ent(\tilde p), \Ent(\tilde p_U) \geq \log |G| - \log K - O(d)$$
and so by Theorem \ref{uni-thm}, 
$$ \trans(\tilde p, p_G), \trans(\tilde p_U, p_G) \ll \log K + d$$
where $p_G$ is the uniform distribution on $G$.  Thus by the triangle inequality
$$ \trans(\tilde p, \tilde p_U) \ll \log K + d.$$
Observe that as $\tilde p, \tilde p_U$ both range in $B$, the shifts needed to transport $\tilde p$ to $\tilde p_U$ range in $B-B$ and so do not encounter the ``wraparound'' effects of the cyclic groups $\Z/(2N_1\Z), \ldots, \Z/(2N_d\Z)$.  Thus we can push this transport bound back to $H+P$ and establish \eqref{transpu} as desired.
\end{proof}

\section{The inverse entropy theorem}\label{iest-sec}

We now prove Theorem \ref{iest}.  We begin with the easy claim (i).  If $X$ is the uniform distribution on a coset of a finite group then $X+X$ is uniformly distributed on another coset of this group, and so $\doubling[X]=1$ as claimed.  Now suppose that $\doubling[X]=1$, thus $\Ent(X_1+X_2)=\Ent(X)$, where $X_1,X_2$ are independent copies of $X$.  Inspecting the proof of Lemma \ref{triv} we conclude that $\Ent( X_1+X_2 ) = \Ent(X_1+X_2|X_2)$, which by the discussion after \eqref{ento} implies that $X_1+X_2$ and $X_2$ are independent, or equivalently that the distribution of $(X_1+X_2|X_2=x)$ is independent of $x \in \range(X)$.  This implies that the probability distribution of $X$ is invariant under translations in $\range(X)-\range(X)$, which quickly implies that $\range(X)-\range(X)$ is a finite subgroup of $G$, and that $X$ is uniformly distributed on a coset of this subgroup, as desired.

Now we prove the more difficult claims (ii), (iii).  We begin with a special case of (iii), in which $Y$ is already uniform.

\begin{proposition}\label{sumtop}  Let $X$ be a $G$-random variable, and let $H+P$ be a coset progression of rank $d$.  Let $U$ be the uniform distribution on $H+P$, and suppose that $\dist_R(X,U) \leq \log K$.  Then $\trans(X,U) \ll_{K,d} 1$.
\end{proposition}

\begin{proof}  By translating $H+P$ if necessary we may assume $0 \in H+P$.  We allow all implied constants to depend on $K,d$.  The basic idea here is to treat $H+P$ as an approximate group, and somehow pass to a ``quotient'' of $G$ by $H+P$.  The reader is encouraged to consider the special case $P=\{0\}$, in which this quotienting idea can be made precise.

Let $S$ be a maximal subset of $G$ with the property that the translates $s+H+P$ of $H+P$ for $s \in S$ are all disjoint.  (For $G$ infinite, the existence of such an $S$ is guaranteed by Zorn's lemma.)  
Clearly the translates $s+(H+P)-(H+P)$ cover $G$.  From this, the disjointness of the $s+H+P$, and the greedy algorithm, we can thus partition $G$ into sets $A_s$ for $s \in S$, where
$$ s+(H+P) \subset A_s \subset s+(H+P)-(H+P).$$
One should view the partition $A_s$ as a crude approximation of the (non-existent) quotient of $G$ by $H+P$.

Take $U$, $X$ to be independent.  From \eqref{ruzsa-def} and the hypothesis $\dist_R(X,U) \leq \log K$ (and the fact that $-U$ is equivalent to a translate of $U$) one has
$$ \Ent(X+U) \leq \frac{1}{2} \Ent(X) + \frac{1}{2} \Ent(U) + O(1).$$
Of course, $\Ent(U) = \log |H+P|$.  Applying \eqref{ent-lower}, we conclude that
$$ \Ent(X) = \log |H+P| + O(1)$$
and
$$ \Ent( X+U ) = \log |H+P| + O(1)$$
and thus
$$ \sum_{s \in S} \sum_{x \in A_s} p_{X+U}(x) \log \frac{1}{p_{X+U}(x)} = \log |H+P| + O(1).$$
On the other hand, if $s \in S$ and $x \in A_s$, one clearly has
\begin{align*}
 p_{X+U}(x) &\leq \frac{1}{|H+P|} \P( X \in s + (H+P) - 2(H+P) ) \\
 &\leq \frac{1}{|H+P|} \P( X+U \in s + 2(H+P) - 2(H+P) )
\end{align*}
and thus
$$ \log \frac{1}{p_{X+U}(x)} \geq \log |H+P| + \log \frac{1}{\P( X+U \in s + 2(H+P) - 2(H+P)  )}.$$
Since $\sum_{s \in S} \sum_{x \in A_x} p_{X+U}(x) = 1$, we conclude that
$$ \sum_{s \in S} \sum_{x \in A_s} p_{X+U}(x) \log \frac{1}{\P( X+U \in s + 2(H+P) - 2(H+P)  )} \leq O(1)$$
or equivalently
$$ \sum_{s \in S} c_s \log \frac{1}{\P( X+U \in s + 2(H + P) - 2(H+P) )} \leq O(1)$$
where $c_s := \P(X+U \in A_s)$.
Observe that $s+2(H+P)-2(H+P)$ can be covered by at most $O(1)$ sets $A_{s'}$, where $s' \in s + 3(H+P) - 3(H+P)$.  (Indeed, all such $A_{s'}$ are disjoint, have cardinality comparable to $|H+P|$, and are contained in $s+4(H+P)-4(H+P)$ which has cardinality $O(|H+P|)$.) Thus, by the pigeonhole principle, for every $s \in S$ there exists $s'(s) \in S \cap (s + 3(H+P) - 3(H+P))$ such that
$$ \P( X+U \in s + 2(H + P) - 2(H+P) ) \ll c_{s'(s)}$$
and thus
\begin{equation}\label{soo}
 \sum_{s \in S} c_s \log \frac{1}{c_{s'(s)}} \leq O(1).
\end{equation}
Let $Y$ be the random variable $Y := s$, where $s$ is the unique $s \in S$ such that $X+U \in A_s$.  Then $X-Y$ takes values in $(H+P)-2(H+P)$.  We now claim that 
\begin{equation}\label{entyo}
\Ent(Y) \leq O(1), 
\end{equation}
or in other words that
\begin{equation}\label{soo2}
\sum_{s \in S} F(c_s) \leq O(1).
\end{equation}
This is almost \eqref{soo}, but we have to replace $s'(s)$ by $s$.  To do this, we let $C > e$ be a large quantity to be chosen later, and split the sum in \eqref{soo2} into three terms: one where $c_{s'(s)} \geq 1/e$, one where $1/e > c_{s'(s)} \geq C c_s$ and one where $c_{s'(s)} \leq C c_s$.

In the first case, observe that there are only $O(1)$ possible values of $s'(s)$; each one of these is associated to $O(1)$ possible values of $s$ (since $s \in s'(s) + 3(H+P)-3(H+P)$ and the $s+H+P$ are disjoint), so the net contribution to \eqref{soo2} here is $O(1)$.

In the second case, we observe from \eqref{fax} that
$$ F(c_s) \leq 2 F(1/C) F(C c_s) \leq 2 F(1/C) F(C c_{s'(s)}),$$ 
so the contribution of this term to \eqref{soo2} is at most
$$ 2 F(1/C) \sum_{s \in S} F(c_{s'(s)}).$$
But each $s'$ can arise from at most $O(1)$ choices of $S$, so we can bound this contribution by at most
$$ \frac{1}{2} \sum_{s' \in S} F(c_{s'}) = \frac{1}{2} \Ent(Y) $$
if $C = O(1)$ is chosen appropriately.

For the third case, we see that
$$ c_s \log \frac{1}{c_s} \leq c_s \log \frac{1}{c_{s'(s)}} + c_s \log C$$
and so by \eqref{soo2} the net contribution of this case is
$$ \leq O(1) + \log C = O(1).$$
Thus $\Ent(Y) \leq \frac{1}{2} \Ent(Y) + O(1)$, and the claim \eqref{entyo} follows.  In particular
$$ \trans(X, X-Y) \ll 1.$$
But by \eqref{ent-sum} one has
$$ \Ent(X-Y) \geq \Ent(X) - O(1) \geq \log |H+P| - O(1).$$
The random variable $X-Y$ ranges in $(H+P)-2(H+P)$, which is a coset progression of rank $d$ and cardinality $O( |H+P| )$.  Applying Corollary \ref{coset-prog} to $X-Y$, we conclude that
$$ \trans(X-Y,U_{(H+P)-2(H+P)}) \ll 1$$
where $U_{(H+P)-2(H+P)}$ is the uniform distribution on $(H+P)-2(H+P)$.  Direct computation shows that
$$ \trans(U_{(H+P)-2(H+P)}, U) \ll 1$$
and the claim follows from another application of the triangle inequality.
\end{proof}

In view of the above proposition, it suffices to show that

\begin{proposition}\label{main}  If $\doubling[X] \leq K$, then there exists a coset progression $H+P$ of rank $O_K(1)$ such that $\dist_R(X,U) \ll_K 1$, where $U$ is the uniform distribution on $H+P$.
\end{proposition}

Indeed, (ii) follows immediately from Proposition \ref{main} and Proposition \ref{sumtop}, while if we are in the situation of (iii), then from Theorem \ref{ese} one has 
\begin{align*}
\dist_R(X,-X) &\leq \dist_R(X,Y) + \dist_R(X,-Y) \\
&\leq 4 \dist_R(X,Y) \\
&\leq 4 \log K
\end{align*}
and thus from \eqref{dubdub} we have $\doubling[X] \leq K^4$.  By Proposition \ref{main}, one can then find a uniform distribution $U$ on a coset progresion $H+P$ on rank $O_K(1)$ such that $\dist_R(X,U) \ll_K 1$, hence $\trans(X,U) \ll_K 1$ by Proposition \ref{sumtop}; meanwhile, from the Ruzsa triangle inequality \eqref{ruzsa-triangle} one has $\dist_R(Y,U) \ll_K 1$ and so $\trans(Y,U) \ll_K 1$, and so by the triangle inequality one has $\trans(X,Y) \ll_K 1$ as claimed.

It remains to establish Proposition \ref{main}.  We begin with an approximate formula for $\Ent(X+Y)-\Ent(X)$ when $X, Y$ are independent.

\begin{lemma}[Sumset entropy increase formula]\label{xysim} Let $X,Y$ be independent.  Then
$$ \sum_{y \in \range(Y)} p_Y(y) \sum_{z \in \range(X+Y)} p_{X+y}(z) \log_+ \frac{p_{X+y}(z)}{p_{X+Y}(z)}  = \Ent(X+Y) - \Ent(X) + O(1),$$
where $\log_+ x := \max(\log x, 0)$.
\end{lemma}

\begin{proof}  To simplify the summation notation, it will be understood throughout that $y \in \range(Y)$ and $z \in \range(X+Y)$. We have
\begin{align*}
\Ent(X+Y) - \Ent(X) &= \sum_{y} p_Y(y) ( \Ent(X+Y) - \Ent(X+y) ) \\
&= \sum_y p_Y(y) \sum_{z} (F(p_{X+Y}(z)) - F(p_{X+y}(z))).
\end{align*}
Since
$$ \sum_{y} p_Y(y) ( p_{X+Y}(z) - p_{X+y}(z) ) = 0$$
for all $z$, we thus have
$$ \Ent(X+Y) - \Ent(X) = \sum_{y} p_Y(y) \sum_z F(p_{X+Y}(z)) + F'(p_{X+Y}(z)) ( p_{X+y}(z) - p_{X+Y}(z) ) - F(p_{X+y}(z)).$$
From \eqref{sub-bound}, the summand is equal to
$$ p_{X+y}(z) \log_+ \frac{p_{X+y}(z)}{p_{X+Y}(z)} + O( p_{X+y}(z) ) + O( p_{X+Y}(z) ).$$
Since
$$ \sum_y p_Y(y) \sum_z p_{X+y}(z) = \sum_y p_Y(y) \sum_z p_{X+Y}(z) = 1,$$
the desired claim follows.
\end{proof}

This leads us to our first structural result on random variables of bounded doubling, namely that they are approximately uniformly distributed in a set that captures the bulk of the entropy.

\begin{proposition}[$X$ is approximately uniformly distributed]\label{xa}  If $\doubling[X] \leq K$, then there exists a non-empty set $A$ of cardinality
\begin{equation}\label{acard}
|A| \asymp_K \exp( \Ent(X) )
\end{equation}
such that
\begin{equation}\label{aprob}
p_X(x) \asymp_K \exp( - \Ent(X) )
\end{equation}
for all $x \in A$. 
\end{proposition}

\begin{proof}  We allow implied constants to depend on $K$.
Write $Z := X_1 + X_2$ for the sum of two independent copies of $X$.  From Lemma \ref{xysim} we have
\begin{equation}\label{jo}
\sum_y p_X(y) \sum_{z} p_{X+y}(z) \log_+ \frac{p_{X+y}(z)}{p_Z(z)} \ll 1,
\end{equation}
where it is understood that $y \in \range(X)$ and $z \in \range(Z)$.

Now let $0 < \eps < 0.1$ be a small constant (depending on $K$) to be chosen later.  From \eqref{jo} we have
$$ \sum_y p_X(y) \sum_{z: p_{X+y}(z) \geq e^{1/\eps} p_Z(z)} p_{X+y}(z) /\eps \ll 1$$
and thus
$$ \sum_z [\sum_{y: p_{X+y}(z) \geq e^{1/\eps} p_Z(z)} p_X(y) p_{X+y}(z)] \ll \eps.$$
Swapping $y$ and $z-y$ (using the identity $p_{X+y}(z) = p_X(z-y)$) we also have
$$ \sum_z \sum_{y: p_X(y) \geq e^{1/\eps} p_Z(z)} p_X(y) p_{X+y}(z) \ll \eps.$$
Also, observe that
\begin{align*}
\sum_z \sum_{y: p_{X+y}(z) \leq \eps p_Z(z) } p_X(y) p_{X+y}(z)
&\leq \sum_z \eps p_Z(z) \sum_y p_X(y) \\
&= \eps
\end{align*}
and similarly
$$ \sum_z \sum_{y: p_X(y) \leq \eps p_Z(z) } p_X(y) p_{X+y}(z) \leq \eps.$$
Finally, we have
\begin{equation}\label{sumo}
\sum_z \sum_y p_X(y) p_{X+y}(z) = 1.
\end{equation}
Putting all these estimates together, (with $\eps$ sufficiently small) we conclude that
\begin{equation}\label{sumx}
 \sum_z \sum_{y: p_X(y), p_{X+y}(z) \asymp p_Z(z)} p_X(y) p_{X+y}(z) > 1/2.
 \end{equation}
From \eqref{sumx}, and the pigeonhole principle, there exists an $z_0 \in \range(Z)$ such that
$$ \sum_{y: p_X(y), p_{X+y}(z_0) \asymp p_Z(z_0)} p_X(y) p_{X+y}(z_0) > p_Z(z_0)/2.$$
The left-hand side can be bounded crudely by
$$ \ll | \{ y: p_X(y) \asymp p_Z(z_0) \} | p_Z(z_0)^2.$$
Thus if we let $A$ denote the set
\begin{equation}\label{adef}
 A := \{ y: p_X(y) \asymp p_Z(z_0) \}
\end{equation}
then 
$$|A| \gg 1 / p_Z(z_0).  $$
Since
$$ 1 \geq \sum_{y \in A} p_X(y) \asymp |A| p_Z(z_0)$$
we conclude that
$$
|A| \asymp 1/p_Z(z_0)$$
and hence by \eqref{adef}
\begin{equation}\label{acard1}
p_X(y) \asymp 1/|A|
\end{equation}
for all $y \in A$.  In particular we have
\begin{equation}\label{Aprob}
\P(X \in A) \asymp 1.
\end{equation}
To conclude the lemma, we need to show that
$$ \log |A|  = \Ent(X) + O(1).$$
We may assume by a limiting argument that the events $X \in A$ and $X \not \in A$ have non-zero probability.
Let $X_1, X_2$ be independent copies of $X$, and let $Y$ be the indicator random variable $Y = 1_{X_1 \in A}$.  Then by \eqref{ento}, \eqref{exy} one has
\begin{equation}\label{hxxy}
 \Ent(X_1+X_2) \geq \Ent(X_1+X_2|Y) = \P(X_1 \in A) \Ent(X_1+X_2|X_1 \in A) + \P(X_1 \not \in A) \Ent(X_1+X_2|X_1 \not \in A).
\end{equation}
Now from Lemma \ref{triv} we have
$$ \Ent(X_1+X_2|X_1 \not \in A) \geq \Ent(X_1|X_1 \not \in A) = \Ent(X|X \not \in A)$$
and
$$ \Ent(X_1+X_2|X_1 \in A) \geq \Ent(X_2) = \Ent(X).$$
On the other hand, since $\sigma[X] \leq K$ by hypothesis, $\Ent(X_1+X_2) \leq \Ent(X)+O(1)$.  Putting all these estimates together, we obtain
$$ \Ent(X)+O(1) \geq \P(X \in A) \Ent(X) + \P(X \not \in A) \Ent(X|X \not \in A)$$
and hence
\begin{equation}\label{hexa}
 \Ent(X|X \not \in A) \leq \Ent(X) + O( 1 / \P( X \not \in A) ).
\end{equation}
A similar argument (swapping the roles of $A$ and its complement) give
\begin{equation}\label{hexa-2}
 \Ent(X|X \in A) \leq \Ent(X) + O( 1 / \P(X \in A) ).
\end{equation}
But from \eqref{acard1}, \eqref{Aprob} one has
\begin{equation}\label{hex} \Ent( X | X \in A ) = \log |A| + O(1).
\end{equation}
Combining this with \eqref{hexa-2}, \eqref{Aprob} we obtain the upper bound
$\log|A| \leq \Ent(X) + O(1)$.

Now we establish the lower bound.  From \eqref{hex}, \eqref{hexa-2} one has
$$ \Ent(X|Y) \leq \P(X \in A) \log |A| + \P(X \not \in A) \Ent(X) + O(1).$$
Since $Y$ is boolean, we have $\Ent(Y) \leq \log 2$.  In particular
$$\Ent(X|Y) = \Ent(X,Y) - \Ent(Y) \geq \Ent(X) - \log 2.$$
Combining this with the previous bound and \eqref{Aprob} we see that
$$ \log |A| \geq \Ent(X) - O( \frac{1}{\P(X \in A)} ) \geq \Ent(X) - O(1)$$
as desired.
\end{proof}

Now we show that $A$ has large additive energy.

\begin{proposition}[$A$ has large energy]\label{xa2} If $\doubling[X] \leq K$, and let $A$ be the set in Proposition \ref{xa}.  Then $|\{ a_1,a_2,a_3,a_4 \in A: a_1+a_2=a_3+a_4 \}| \gg_K |A|^3$.
\end{proposition}

\begin{proof}  Again, we allow implied constants to depend on $K$.
Let $X_1, X_2$ be independent copies of $X$, and let $Y_1, Y_2$ be the indicators $Y_i = 1_{X_i \in A}$.  We have
\begin{align*}
\Ent(X) + O(1) &\geq
\Ent(X_1+X_2) \\
&\geq \Ent(X_1+X_2|Y_1,Y_2) \\
&= \P(X_1 \in A) \P(X_2 \in A) \Ent(X_1+X_2|X_1,X_2 \in A) \\
&\quad + \P(X_1 \in A) \P(X_2 \not \in A) \Ent(X_1+X_2|X_1 \in A; X_2 \not \in A) \\
&\quad + \P(X_1 \not \in A) \P(X_2 \in A) \Ent(X_1+X_2|X_1 \not \in A; X_2 \in A) \\
&\quad + \P(X_1 \not \in A) \P(X_2 \not \in A) \Ent(X_1+X_2|X_1,X_2 \not \in A)
\end{align*}
and
\begin{align*}
\Ent(X) &\leq \frac{1}{2} \Ent(X_1,Y_1) + \frac{1}{2} \Ent(X_2,Y_2) \\
&\leq \frac{1}{2} \Ent(X_1|Y_1) + \frac{1}{2} \Ent(X_2|Y_2) + \log 2 \\
&= \P(X_1 \in A) \P(X_2 \in A) (\frac{1}{2} \Ent(X_1|X_1 \in A) + \frac{1}{2} \Ent(X_2|X_2 \in A)) \\
&\quad + \P(X_1 \in A) \P(X_2 \not \in A) (\frac{1}{2} \Ent(X_1|X_1 \in A) + \frac{1}{2} \Ent(X_2|X_2 \not \in A)) \\
&\quad + \P(X_1 \not \in A) \P(X_2 \in A) (\frac{1}{2} \Ent(X_1|X_1 \not \in A) + \frac{1}{2} \Ent(X_2|X_2 \in A)) \\
&\quad + \P(X_1 \not \in A) \P(X_2 \not \in A) (\frac{1}{2} \Ent(X_1|X_1 \not \in A) + \frac{1}{2} \Ent(X_2|X_2 \not \in A)) \\
&\quad + \log 2.
\end{align*}
Now applying Lemma \ref{triv} we have
$$ \Ent(X_1+X_2| X_1 \in A_1, X_2 \in A_2 ) \geq \frac{1}{2} \Ent( X_1 | X_1 \in A_1 ) + \frac{1}{2} \Ent(X_2 | X_2 \in A_2 )$$
for any sets $A_1,A_2$.  Inserting this into the first estimate and then subtracting from the second, we conclude in particular
that
$$ \P(X_1 \in A) \P(X_2 \in A) (\Ent(X_1+X_2|X_1,X_2 \in A) - \frac{1}{2} \Ent(X_1|X_1 \in A) - \frac{1}{2} \Ent(X_2|X_2 \in A))
\leq O(1)$$
and hence (by \eqref{Aprob})
$$ \Ent(X_1+X_2|X_1,X_2 \in A) \leq \Ent(X|X \in A) + O(1).$$
Let $X'$ be the random variable $X$ conditioned to the event $X \in A$.  Then $X'$ now obeys the hypotheses of Proposition \ref{xa} (with $K$ replaced by a larger but still bounded quantity).  Repeating the derivation of \eqref{sumx}, we conclude
$$
 \sum_z \sum_{y: p_{X'}(y), p_{X'+y}(z) \asymp p_{Z'}(z)} p_{X'}(y) p_{X'+y}(z) \geq 1/2$$
where $Z'$ is the sum of two independent copies of $X'$.  Observe that the summand here vanishes unless $y, z-y \in A$, in which case the summand is $O( 1/|A|^2 )$.  Setting $x := z-y$, we conclude that
\begin{equation}\label{un}
|\{ x, y \in A: p_{Z'}(x+y) \asymp 1/|A| \}| \gg |A|^2.
\end{equation}
Since $p_{Z'}(z) \asymp |\{ a,a' \in A: a+a' = z \}|/|A|^2$, the claim follows.
\end{proof}

From Proposition \ref{xa2} (or \eqref{un}) one can find $E \subset A \times A$ with $|E| \gg_K |A|^2$ such that $|A \stackrel{E}{+} A| \ll_K |A|$. Applying the Balog-Szemer\'edi-Gowers theorem (Lemma \ref{bsg}) we conclude there exists a subset $A'$ of $A$ with $|A'+A'| \asymp_K |A'| \asymp_K |A|$.  Applying Freiman's theorem in an arbitrary additive group (Theorem \ref{grf}) we conclude

\begin{corollary}[Concentration on a coset progression]\label{cos} If $\doubling[X] \leq K$, then there exists a coset progression $H+P$ of rank $O_K(1)$ and cardinality
$$ |H+P| \asymp_K \exp(\Ent(X) )$$
such that
\begin{equation}\label{pax}
 p_X( x ) \asymp_K \exp(-\Ent(X) )
\end{equation}
for $\gg_K |H+P|$ elements $x$ of $H+P$.  \qed
\end{corollary}

Now we are ready to prove Proposition \ref{main}. 
Let $X_1, X_2$ be independent copies of $X$, let $Y_2$ be the indicator of the event $X_2 \in H+P$, and let $X'_2$ be the conditioning of $X_2$ to the event $X_2 \in H+P$.  Let $U$ be a uniform distribution on $H+P$, taken to be independent of $X_1, X'_2$. From Corollary \ref{cos} we have 
\begin{equation}\label{sor}
\P( X_2 \in H+P ) \asymp 1
\end{equation}
and 
\begin{equation}\label{did}
\Ent(X'_2) = \Ent(X) + O(1) = \log |H+P| + O(1).
\end{equation}
(The lower bound on $\Ent(X'_2)$ follows from \eqref{pax} and the definition of entropy; the upper bound follows from Jensen's inequality, Lemma \ref{jensen}.)

Next, observe that
\begin{align*}
 \Ent(X_1 + X_2) &\geq \Ent(X_1+X_2|Y_2)\\
&= \P( X_2 \in H+P ) \Ent(X_1+X_2|X_2 \in H+P) + \P( X_2 \not \in H+P) \Ent(X_1+X_2|X_2 \not \in H+P) \\
&\geq \P( X_2 \in H+P ) \Ent(X_1+X'_2) + \P( X_2 \not \in H+P) \Ent(X_1) \\
&= \Ent(X_1) + \P( X_2 \in H+P ) (\Ent(X_1+X'_2) - \Ent(X_1)).
\end{align*}
By the hypothesis $\sigma[X] \leq K$, one has $\Ent(X_1+X_2) \leq \Ent(X_1)+O(1)$.  Applying \eqref{sor}, one concludes that
$$ \Ent(X_1+X'_2) \leq \Ent(X_1) + O(1).$$
From this and \eqref{did} we see that $\dist_R(X_1,-X'_2) = O(1)$.
Meanwhile, from Jensen's inequality one has
\begin{align*}
\Ent(X'_2 - U) &\leq \log |(H+P)-(H+P)| \\
&\leq \log |H+P| + O(1),
\end{align*}
which implies that $\dist_R(X'_2,U) = O(1)$.  Applying the triangle inequality \eqref{ruzsa-triangle} we obtain the claim.

The proof of Proposition \ref{main}, and thus Theorem \ref{iest}, is now complete.

\section{Proof of Theorem \ref{abb}}\label{abb-sec}

We now prove Theorem \ref{abb}.  The basic idea is to get enough control on $X$ that one can find a ``smooth'' direction in which to approximate the discrete random variable by a continuous one.

Fix $\eps$, and assume $X$ to be a $G$-random variable with $\Ent(X)$ sufficiently large depending on $\eps$.  We assume for contradiction that the claim failed, thus (after adjusting $\eps$ slightly)
$$ \Ent( X_1 + X_2 ) < \Ent( X ) + \frac{1}{2} \log 2 - \eps$$
We can then apply Theorem \ref{iest}(ii) and express $X = U+Z$, where $U$ is the uniform distribution in a coset progression $H+P$ of rank $O(1)$ and cardinality $O( \exp( \Ent(X) ) )$, and $\Ent(Z) = O(1)$.  Since $G$ is torsion-free, the $H$ component of the coset progression is trivial, thus $U$ is just the uniform distribution on $P$.

Since $\Ent(Z) = O(1)$, we have
$$ \sum_z p_Z(z) \log \frac{1}{p_Z(z)} = O(1).$$
Let $0 < \delta < 1/2$ be a small number depending on $\eps$ to be chosen later.  Let $A := \{z: p_Z(z) \geq \delta \}$, thus $|A| \leq  1/\delta$.  Also, since
$$ \sum_z p_Z(z) \log \frac{1}{p_Z(z)} \geq (\log \frac{1}{\delta}) \sum_{z \not \in A} p_Z(z) = (\log \frac{1}{\delta})  \P( Z \not \in A )$$
we see that
$$ \P( Z \not \in A ) \ll \frac{1}{\log \frac{1}{\delta}}.$$
This implies that
$$ \Ent(1_{Z \in A}) \ll \frac{\log\log \frac{1}{\delta}}{\log \frac{1}{\delta}}.$$
This implies from \eqref{e-sob} that
$$ \Ent(X|1_{Z \in A}) \geq \Ent(X) - O( \frac{\log\log \frac{1}{\delta}}{\log \frac{1}{\delta}} )$$
and thus
\begin{equation}\label{ea}
\Ent(X|Z \in A) \P(Z \in A) + \Ent(X|Z \not \in A) \P( Z \not \in A) \geq \Ent(X) - O(\frac{\log\log \frac{1}{\delta}}{\log \frac{1}{\delta}}).
\end{equation}
If we let $X_1,Z_1$ and $X_2,Z_2$ be independent copies of $X,Z$, then we have
\begin{align*}
\Ent(X_1+X_2) &\geq \Ent(X_1+X_2|1_{Z_1 \in A},1_{Z_2 \in A}) \\
&\geq \Ent(X_1+X_2|Z_1, Z_2 \in A) \P(Z \in A)^2 \\
&\quad + \Ent(X_1+X_2|Z_1 \in A; Z_2 \not \in A) \P(Z \in A) (1 - \P(Z \in A)) \\
&\quad + \Ent(X_1+X_2|Z_2 \in A; Z_1 \not \in A) \P(Z \in A) (1 - \P(Z \in A)) \\
&\quad + \Ent(X_1+X_2|Z_1,Z_2 \not \in A) (1 - \P(Z \in A))^2.
\end{align*}
From Lemma \ref{triv} one has
\begin{align*}
 \Ent(X_1+X_2|Z_2 \in A; Z_1 \not \in A) &\geq \Ent(X|Z \in A) \\
 \Ent(X_1+X_2|Z_1 \in A; Z_2 \not \in A) &\geq \Ent(X|Z \not \in A) \\
 \Ent(X_1+X_2|Z_1,Z_2 \not \in A) &\geq \Ent(X|Z \not \in A) 
\end{align*}
so from \eqref{ea} we see that
$$
\Ent(X_1+X_2) \geq \Ent(X) - O( \frac{\log \log \frac{1}{\delta}}{\log \frac{1}{\delta}}) + (\Ent(X_1+X_2|Z_1,Z_2 \in A) - \Ent(X|Z \in A)) \P(Z \in A)^2.$$
Thus, by taking $\delta$ small enough, it will suffice to show that
\begin{equation}\label{entropy}
 \Ent(X'_1+X'_2) \geq \Ent(X') + \frac{1}{2} \log 2 - \eps/2
\end{equation}
(say), where $X' := (X|Z \in A)$ and $X'_1, X'_2$ are independent copies of $X'$.  

Observe that $X'$ ranges in the set $A+P$; since $|A| \ll_\delta 1$, we may place $A+P$ inside a progression $Q$ of rank $O_\delta(1)$ and size $O_\delta( \exp(\Ent(X) ) ) = O_\delta(|P|)$; by \cite[Theorem 1.9]{tv-john}, we may assume that $Q$ is $4$-proper, thus
$$ Q = \{ a + n_1 v_1 + \ldots + n_d v_d: n_1 \in [0,N_1), \ldots, n_d \in [0,N_d) \}$$
for some $d = O_\delta(1)$ and integers $N_1,\ldots,N_d$, and the sums $a+n_1 v_1 + \ldots + n_d v_d$ for $n_1 \in [0,4N_1),\ldots, n_d \in [0,4N_d)$ are all distinct.  Using a Freiman isomorphism (see e.g. \cite[Section 5.3]{tao-vu}), we may thus identify $Q$ with the box $B := [0,N_1) \times \ldots \times [0,N_d)$ in $\Z^d$.  Without loss of generality we may assume that $N_1 \geq \dots \geq N_d$, in particular $N_1 \gg_\delta \exp( \Ent(X)/d )$.

Let $X''$ be the counterpart of $X'$ in $B$, thus $X''$ is Freiman isomorphic to $X'$.  
Since $X' = (U+Z|Z \in A)$, with $U$ the uniform distribution on $P$, we see that
\begin{equation}\label{pax2}
 p_{X''}(x) \ll_\delta 1/|P| \asymp_\delta 1/|B|
\end{equation}
for all $x \in B$.

Now we establish some ``smoothness'' in the probability distribution function $p_{X''_1+X''_2}$ in some short direction, as measured using the total variation metric \eqref{tv-def}.

\begin{lemma}[Smoothness of $p_{X''_1+X''_2}$]  Let $0 < \mu < 1$.  Then, if $\Ent(X)$ is sufficiently large depending on $\mu, \delta$, there exists $r \in [1, N_1) \times \{0\} \times \dots \times \{0\}$ with $|r| \ll_\delta \mu^{-O_\delta(\mu^{-2})}$ such that
\begin{equation}\label{b-plus}
\dist_{TV}( X''_1+X''_2 + r, X''_1+X''_2 ) \ll_\delta \mu.
\end{equation}
\end{lemma}

\begin{proof} For this lemma it is convenient to embed $B$ and $X''$ inside the finite group
$$G' := \Z/3N_1\Z \times \ldots \times \Z/3N_d\Z,$$
thus $p_{X''}$ is now a function on $G'$.  The left-hand side of \eqref{b-plus} can thus be written as
$$ \sum_{x \in G'} |p_{X''}*p_{X''}(x+r) - p_{X''}*p_{X''}(x)|$$
where $p_{X''}*p_{X''}$ is the convolution
$$ p_{X''}*p_{X''}(x) := \sum_{y \in G'} p_{X''}(y) p_{X''}(x-y).$$
We introduce the Fourier coefficients
$$ \hat p_{X''}(\chi) := \sum_{x \in G'} p_{X''}(x) \overline{\chi(x)}$$
for all characters $\chi: G' \to S^1$ in the Pontraygin dual $\hat G'$ of $G'$. 
From Plancherel's theorem and \eqref{pax2} one has
\begin{equation}\label{mud}
\begin{split}
\sum_{\chi \in \hat G'} |\hat p_{X''}(\chi)|^2 &= |G'| \sum_{x \in G'} |p_{X''}(x)|^2 \\
&\ll_\delta 1.
\end{split}
\end{equation}
Thus, if we set
\begin{equation}\label{lamd}
\Lambda := \{ \chi \in \hat G': |\hat p_{X''}(\chi)| \geq \mu \}
\end{equation}
then
\begin{equation}\label{lama}
 |\Lambda| \ll_{\delta} \mu^{-2}.
\end{equation}
If $\Ent(X)$ is large, then $N_1 \ldots N_d$ is large.  By the Kronecker approximation theorem (see e.g. \cite[Corollary 3.25]{tao-vu}), if $\Ent(X)$ is large enough depending on $\delta,\mu$, we may thus find $r \in [1, N_1) \times \{0\} \times \dots \times \{0\}$ with $|r| \ll_\delta \mu^{-O_\delta(\mu^{-2})}$ such that
\begin{equation}\label{chira}
 |\chi(r) - 1| \leq \mu^2
\end{equation}
for all $\chi \in \Lambda$.

Fix this $r$.  From the Fourier inversion formula one has
$$ p_{X''}*p_{X''}(x) = \frac{1}{|G'|} \sum_{\chi \in \hat G'} \hat p_{X''}(\chi)^2 \chi(x)$$
and thus
$$ p_{X''}*p_{X''}(x+r) - p_{X''}*p_{X''}(x) = \frac{1}{|G'|} \sum_{\chi \in \hat G'} \hat p_{X''}(\chi)^2 (\chi(r)-1) \chi(x).$$
By Plancherel's theorem, we conclude
$$ \sum_{x \in G'} |p_{X''}*p_{X''}(x+r) - p_{X''}*p_{X''}(x)|^2 = \frac{1}{|G'|} \sum_{\chi \in \hat G'} |\hat p_{X''}(\chi)|^4 |\chi(r)-1|^2.$$
From \eqref{lama}, \eqref{chira} the contribution of the terms with $\chi \in \Lambda$ are $O_\delta(\mu^2/|G'|)$; by \eqref{lamd}, \eqref{mud}, the contribution of the terms with $\chi \not \in \Lambda$ are also $O_\delta(\mu^2/|G'|)$.  We thus have
$$ \sum_{x \in G'} |p_{X''}*p_{X''}(x+r) - p_{X''}*p_{X''}(x)|^2 \ll_\delta \mu^2/|G'|$$
and the claim follows from the Cauchy-Schwarz inequality.
\end{proof}

Let $0 < \mu < 1$ be a small number (depending on $\delta,\eps$) to be chosen later, and let $r$ be as in the above lemma.  We can write $r = mr'$, where $m \geq 1$ is an integer with 
\begin{equation}\label{mosey}
m \ll_\delta \mu^{-O_\delta(\mu^{-2})}, 
\end{equation}
and $r'$ is irreducible in $\Z^d$.  Applying yet another Freiman isomorphism, we may then map $B$ to the integers so that $r$ maps to $m$.  If $X'''$ is the image of $X''$ under this isomorphism, then $X'''$ is isomorphic to $X'$, ranges over at most $|B|$ values, and
\begin{equation}\label{paxum}
\dist_{TV}( X'''_1 + X'''_2 + m, X'''_1 + X'''_2 ) \ll_\delta \mu
\end{equation}
Our task is now to show that
\begin{equation}\label{saxon}
\Ent(X'''_1 + X'''_2) \geq \Ent(X''') + \frac{1}{2} \log 2 - \eps/2.
\end{equation}

To motivate the general argument later, let us first consider the simpler \emph{irreducible case} when $m=1$, thus the distribution function $p_{X'''_1 + X'''_2}(x)$ looks ``locally smooth''.  To exploit this, let $U$ be the \emph{continuous} random variable uniformly distributed in $[0,1]$, independent of $X'''$.   Recall that the continuous Shannon entropy $\Ent_\R(V)$ of a random variable on $\R$ with distribution $p_V(x)\ dx$ is given by 
$$ \Ent_\R(V) := \int_\R F(p_V(x))\ dx.$$
A short computation then relates the continuous Shannon entropy to the discrete entropy:
$$ \Ent_\R( X''' + U ) = \Ent(X''').$$
Now let us look at the continuous variable $V := X'''_1 + U_1 + X'''_2 + U_2$.  We write
$$ \Ent_\R(V) = \log |P| + \int_\R F(p_V(x)) -  p_V(x) \log |P|\ dx$$
where $p_V$ is the density function of $V$.
 
Observe that for any $x \in [n,n+1]$, the density function $p_V(x)$ of $V$ at $x$ is equal to some average of $p_{X'''_1+X'''_2}(n)$, $p_{X'''_1+X'''_2}(n-1)$, thus
$$ p_V(x) = p_{X'''_1+X'''_2}(n) + O( g(n) )$$
where
$$ g(n) := |p_{X'''_1+X'''_2}(n) - p_{X'''_1+X'''_2}(n-1)|.$$
In particular $p_V(x) \ll_\delta 1/|P|$, by \eqref{pax2}.  Using the elementary estimate
$$ F(b) - b \log |P| = F(a) - a \log |P| + O_\delta( (\frac{\mu}{|P|} + |b-a|) \log \frac{1}{\mu} )$$
when $a, b \ll_\delta 1/|P|$ (which arises from the fact that $F'(c) = \log |P| + O_\delta( \log \frac{1}{\mu} )$ for $\mu/|P| \leq c \ll_\delta 1/|P|$), we thus have
$$ \Ent_\R(V) = \log |P| + \sum_{n \in \range(X'''_1+X'''_2)+\{0,1\}} F(p_{X'''_1+X'''_2}(n)) -  p_{X'''_1+X'''_2}(n) \log |P| +
O_\delta( (\frac{\mu}{|P|} + g(n)) \log \frac{1}{\mu} ).$$
From \eqref{paxum}, $\sum_{n \in \range(X'''_1+X'''_2)+\{0,1\}} (\frac{\mu}{|P|} + g(n)) \ll_\delta \mu$, and thus
$$ \Ent_\R(V) = \Ent(X'''_1+X'''_2) + O_\delta( \mu \log \frac{1}{\mu} ).$$
On the other hand, from Shannon's inequality \eqref{st}, we have
$$ \Ent_\R(V) \geq \Ent_\R(X'''+U) + \frac{1}{2} \log 2.$$
Putting all this together, one obtains
$$ \Ent(X'''_1+X'''_2) \geq \Ent(X''') + \frac{1}{2} \log 2 - O_\delta( \mu \log \frac{1}{\mu} )$$
and the claim \eqref{saxon} follows by taking $\mu$ small enough.

Now we return to the general case, when $m$ is not necessarily $1$.  We then introduce the random variable $W := X''' \hbox{ mod } m \in \Z/m\Z$, and define $W_1, W_2$ analogously.  Then
$$ 
\Ent(X'''_1 + X'''_2) \geq \Ent(X'''_1 + X'''_2|W_1).$$
Observe that $X'''_1+X'''_2$ and $W_1$ determine $W_2$, and thus
$$ \Ent(X'''_1 + X'''_2|W_1) = \Ent(W_2) + \Ent(X'''_1 + X'''_2|W_1,W_2).$$
We can write
$$ \Ent(X'''_1 + X'''_2|W_1,W_2) = \sum_{w_1, w_2 \in \Z/m\Z} p_{W_1}(w_1) p_{W_2}(w_2) \Ent( X_{1,w_1} + X_{2,w_2} )$$
where for $i=1,2$, $X_{i,w_i}$ is the $\Z$-random variable $(X_i - w_i)/m$ conditioned to the event $W_i = w_i$.  Meanwhile,
$$ \Ent(X''') = \Ent(W) + \sum_{w_1 \in \Z/m\Z} p_{W_1}(w_1) \Ent(X_{1,w_1})$$
and similarly with the $1$ index replaced by $2$, thus
$$ \Ent(X''') = \Ent(W) + \sum_{w_1,w_2 \in \Z/m\Z} p_{W_1}(w_1)p_{W_2}(w_2) \frac{1}{2} (\Ent(X_{1,w_1}) + \Ent(X_{2,w_2}));$$
putting all this together, we see that to show \eqref{saxon}, it will suffice to show that
\begin{equation}\label{saxon-2}
\sum_{w_1, w_2 \in \Z/m\Z} p_{W_1}(w_1) p_{W_2}(w_2) [ \Ent( X_{1,w_1} + X_{2,w_2} ) - \frac{1}{2} (\Ent(X_{1,w_1}) + \Ent(X_{2,w_2}))] \geq \frac{1}{2} \log 2 - \eps/2.
\end{equation}
From Lemma \ref{triv}, the expression in brackets is non-negative.  Thus we may restrict the sum to a smaller range of $w_1, w_2$; more specifically, we will restrict to the range where
\begin{equation}\label{powa}
 p_{W_1}(w_1), p_{W_2}(w_2) \geq \mu / m.
\end{equation}
On this range, we have from \eqref{pax2}, \eqref{mosey} that
$$ p_{X_{i,w_i}}(x) \ll_\delta \mu^{-1} m/|P|$$
for all $i=1,2$ and $x \in \Z$; also observe that $X_{i,w_i}$ takes on at most $O(|P|/m)$ values.  From \eqref{paxum} we have
\begin{equation}\label{powc}
\sum_{w_i \in \Z/m\Z} p_{W_i}(w_i) \dist_{TV}(X_{i,w_i}+1, X_{i,w_i}) \ll_\delta \mu
\end{equation}
for $i=1,2$.  We will now restrict $w_1,w_2$ further, by imposing the additional restriction
\begin{equation}\label{powb}
\dist_{TV}(X_{i,w_i}+1, X_{i,w_i}) \leq \mu^{1/2}
\end{equation}
for $i=1,2$.

Now we repeat the arguments from the $m=1$ case.  Let $U_1, U_2$ be independent copies of the uniform distribution of $[0,1]$, then as before we have
$$ \Ent_\R( X_{i,w_i} + U_i ) = \Ent( X_{i,w_i} )$$
and
$$ \Ent_\R(X_{1,w_1} + U_1 + X_{2,w_2} + U_2 ) = \Ent( X_{1,w_1} + X_{2,w_2} ) + O_\delta( \mu^{1/2} \log \frac{1}{\mu} );$$
applying \eqref{st}, we conclude
$$ \Ent( X_{1,w_1} + X_{2,w_2} ) - \frac{1}{2} (\Ent(X_{1,w_1}) + \Ent(X_{2,w_2})) \geq \frac{1}{2} \log 2 - O_\delta( \mu^{1/2} \log \frac{1}{\mu} ).$$
To conclude the proof of \eqref{saxon-2}, it thus suffices (on taking $\mu$ small enough) to show that
$$ \sum_{w_1,w_2} p_{W_1}(w_1) p_{W_2}(w_2) \geq 1 - \eps/4$$
(say), where $w_1,w_2$ range over all pairs in $\Z/m\Z$ obeying \eqref{powa}, \eqref{powb}.  But the contribution of those $w_1,w_2$ that fail to obey \eqref{powa} is $O(\mu)$, while from \eqref{powc} the contribution of the $w_1,w_2$ that fail to obey \eqref{powb} is $O(\mu^{1/2})$, and the claim follows by taking $\mu$ small enough.

\appendix

\section{Basic properties of entropy}\label{shannon-sec}

The function $F(x) := x \log \frac{1}{x}$ defined in \eqref{F-def} has first derivative
$$ F'(x) = \log \frac{1}{x} - 1$$
and second derivative
$$ F''(x) = -\frac{1}{x}$$
for $x>0$; from this one easily concludes that $F$ is concave on $\R^+$, and increasing for $x<1/e$.  In particular, we have the upper bound
\begin{equation}\label{upper}
F(x) \leq F(1/e) = 1/e,
\end{equation}
the inequality
\begin{equation}\label{sublinear}
F(y) \leq F(x) + F'(x) (y-x)
\end{equation}
for all $y \geq 0$ and $x > 0$, and the subadditivity property
\begin{equation}\label{subadd}
F(x+y) \leq F(x)+F(y)
\end{equation}
for all $x,y \geq 0$. In particular we have the triangle inequality
\begin{equation}\label{triangle}
|F(a)-F(b)| \leq F(|a-b|)
\end{equation}
for $0 \leq a, b \leq 1/e$.  From the identity
$$ F(x) + F'(x) (y-x) - F(y) = y ( \frac{x}{y} - 1 - \log \frac{x}{y} )$$
and \eqref{sublinear} we obtain the bound
\begin{equation}\label{sub-bound}
F(x) + F'(x) (y-x) - F(y) = y \log_+ \frac{y}{x} + O(x) + O(y)
\end{equation}
where $\log_+ x := \max(\log x, 0)$.  Finally, from the identity
$$ F(ax) = F(a) F(x) (\frac{1}{\log \frac{1}{a}} + \frac{1}{\log \frac{1}{x}} )$$
we see that
\begin{equation}\label{fax}
F(ax) \leq 2 F(a) F(x)
\end{equation}
whenever $0 \leq a, x \leq 1/e$.

\begin{lemma}[Jensen bound]\label{jensen}  Let $A$ be a finite set, and let $X$ be an $A$-random variable.  Then $\Ent(X) \leq \log |A|$.  Furthermore, if
$$ \Ent(X) \geq \log |A| - \log K$$
for some $K \geq 1$, then
$$
\sum_{k=1}^\infty 2^k \P( X \in A_k ) \ll 1 + \log K
$$ 
where
$$ A_k := \{ x \in A: 2^{2^{k-1}} \leq p_X(x) |A| \leq 2^{2^k} \}.$$
\end{lemma}

\begin{proof}  For the first bound, we observe that
\begin{align*}
\Ent(X) &= \sum_{x \in A} F(p_X(x)) \\
&\leq \sum_{x \in A} F( \frac{1}{|A|} ) + F'(\frac{1}{|A|}) (p_X(x) - \frac{1}{|A|} ) \\
&= \log |A|
\end{align*}
as required.  Similarly, if $\Ent(X) \geq \log |A| - \log K$, then the above argument shows that
$$
\sum_{x \in A} F( \frac{1}{|A|} ) + F'(\frac{1}{|A|}) (p_X(x) - \frac{1}{|A|} ) - F(p_X(x)) \leq \log K.
$$
From \eqref{sublinear}, the summand is non-negative; from \eqref{sub-bound}, the summand is $p_X(x) \log(|A| p_X(x)) + O( p_X(x) )$
for $p_X(x) \geq 1/|A|$, and the claim follows by decomposing the $x$ variable into the sets $A_k$.
\end{proof}

Let $X$ be a discrete random variable, and let $E$ be an event which occurs with positive probability.  Then we can define the \emph{conditioned random variable} $(X|E)$ by restricting the underlying probability measure to $E$ (and then dividing out by $\P(E)$ to recover the normalisation), thus
$$ p_{(X|E)}(x) = \P( x \in X \wedge E ) / \P(E).$$
In the special case where $E$ is an event of the form $X \in A$ for some set $A$, we conclude that
$$ p_{(X|X \in A)}(x) = \frac{1_A(x) p_X(x)}{\sum_{y \in A} p_X(y)}.$$
Given two random variables $X, Y$ (not necessarily independent), we define the \emph{conditional entropy} $\Ent(X|Y)$ by the formula
\begin{equation}\label{exy}
 \Ent(X|Y) := \sum_{y \in \range(Y)} p_Y(y) \Ent(X|Y=y).
\end{equation}
A standard computation reveals the identity
\begin{equation}\label{eident}
\Ent(X|Y) = \Ent(X,Y) - \Ent(Y),
\end{equation}
and in particular
\begin{equation}\label{eo}
\Ent(X|Y) = \Ent(X,Y|Y).
\end{equation}
Meanwhile, one has the total probability formula
$$
p_X(x) = \sum_{y \in \range(Y)} p_Y(y) p_{(X|Y=y)}(x);$$
comparing this with \eqref{exy} and Jensen's inequality using the concavity of $F$ we conclude that
\begin{equation}\label{ento}
\Ent(X|Y) \leq \Ent(X)
\end{equation}
with equality if and only if $(X|Y=y) \equiv X$ for all $y \in \range(Y)$, or in other words if $X$ and $Y$ are independent.  From this and \eqref{eident} we conclude that
\begin{equation}\label{ent-sum-}
\Ent(X,Y) \leq \Ent(X) + \Ent(Y)
\end{equation}

We say that a discrete random variable $Y$ is \emph{determined} by another discrete random variable $X$, if one has $Y=f(X)$ for some function $f: \range(X) \to \range(Y)$.  From the subadditivity property \eqref{subadd} we see that
\begin{equation}\label{entyx}
\Ent(Y) \leq \Ent(X)
\end{equation}
whenever $X$ determines $Y$.  For instance, since $(X,Y)$ determines both $X$ and $Y$, 
\begin{equation}\label{ent-0}
\Ent(X), \Ent(Y) \leq \Ent(X,Y),
\end{equation}
and hence by \eqref{eident}, \eqref{ento}
\begin{equation}\label{e-sob}
\Ent(X) - \Ent(Y) \leq \Ent(X|Y) \leq \Ent(X).
\end{equation}
If $X$ determines $Y$, then $X$ and $(X,Y)$ determine each other, and so $\Ent(X,Y) = \Ent(X)$; in particular,
\begin{equation}\label{yush}
 \Ent(X|Y) = \Ent(X) - \Ent(Y)
\end{equation}
and $\Ent(Y|X)=0$.

We have the following useful inequality:

\begin{lemma}[Submodularity inequality]\label{submodularity}  If $X_0,X_1,X_2,X_{12}$ are random variables such that $X_1$ and $X_2$ each determine $X_0$, and $(X_1,X_2)$ determine $X_{12}$, then 
$$ \Ent(X_{12}) + \Ent(X_0) \leq \Ent(X_1) + \Ent(X_2).$$
\end{lemma}

\begin{proof} By \eqref{yush}, \eqref{e-sob} it suffices to show that
$$ \Ent(X_{12}|X_0) \leq \Ent(X_1|X_0) + \Ent(X_2|X_0).$$
By \eqref{exy} it suffices to show that
$$ \Ent(X_{12}|X_0=x_0) \leq \Ent(X_1|X_0=x_0) + \Ent(X_2|X_0=x_0)$$
for all $x_0 \in \range(X_0)$.  But by hypothesis, $(X_1|X_0=x_0)$ and $(X_2|X_0=x_0)$ determine $(X_{12}|X_0=x_0)$, and the claim then follows from \eqref{ent-sum-} and \eqref{entyx}.
\end{proof}

As a special case of Lemma \ref{submodularity} (and \eqref{eo}) we see that
\begin{equation}\label{fsqueeze}
\Ent(Y|Z) \leq \Ent(X|Z)
\end{equation}
whenever $(X,Z)$ determines $Y$.  Similarly, we have
\begin{equation}\label{ent-subadd}
\Ent(X,Y|Z) \leq \Ent(X|Z) + \Ent(Y|Z)
\end{equation}
for any $X,Y,Z$, with equality if and only if $(X|Z=z)$ and $(Y|Z=z)$ are independent for all $z \in \range(Z)$, i.e. if $X$ and $Y$ are \emph{conditionally independent} relative to $Z$.

\end{document}